\documentclass[11pt]{article}

\usepackage[dvips,a4paper,hmargin=2.75cm,vmargin=3cm]{geometry}
\usepackage[utf8]{inputenc}
\usepackage{hyperref}
\usepackage{amssymb,amsmath,amsthm,mathtools}
\usepackage{enumitem}
\usepackage[mathcal]{euscript}
\usepackage[T1]{fontenc}
\usepackage{lmodern}
\usepackage{fontawesome}
\usepackage{bm}
\usepackage{microtype}
\usepackage{graphicx}
\usepackage[normalem]{ulem}
\usepackage{authblk}
\usepackage[yyyymmdd,24hr]{datetime}
\usepackage[small]{titlesec}
\usepackage{cite}
\usepackage{xcolor}

\hypersetup{colorlinks=true,linkcolor=blue,citecolor=blue,
			pdfinfo={
			Title   = {The Dirichlet problem for perturbed Stark operators
						in the half-line},
			Author  = {Julio H. Toloza, Alfredo Uribe},
			Subject = {Manuscript}}
			}

\titleformat*{\subsection}{\bfseries\itshape}

\providecommand{\keywords}[1]
{
\noindent{\small
	\textbf{Keywords:} #1}
}

\providecommand{\msc}[1]
{
\noindent{\small
	\textbf{2010 MSC:} #1}
}

\newtheorem{theorem}{Theorem}[section]

\newtheorem{lemma}[theorem]{Lemma}

\newtheorem{proposition}[theorem]{Proposition}
\newtheorem*{thm}{Theorem}

\theoremstyle{definition}

\newtheorem{examplex}[theorem]{Remark}
\newenvironment{remark}
{\pushQED{\qed}\examplex}
	{\popQED\endexamplex}

\newcommand{\R}{{\mathbb R}}
\newcommand{\N}{{\mathbb N}}

\newcommand{\C}{{\mathbb C}}
\newcommand{\K}{{\mathbb K}}

\newcommand{\cA}{\mathcal{A}}
\newcommand{\cB}{\mathcal{B}}
\newcommand{\cU}{\mathcal{U}}

\newcommand{\cE}{\mathcal{E}}

\newcommand{\cV}{\mathcal{V}}

\newcommand{\nn}{\nonumber}

\DeclareMathOperator{\ai}{Ai}
\DeclareMathOperator{\bi}{Bi}
\DeclareMathOperator{\ci}{Ci}
\DeclareMathOperator{\ch}{ch}

\DeclareMathOperator*{\re}{Re}
\DeclareMathOperator*{\im}{Im}

\newcommand{\abs}[1]{\left\lvert #1 \right\rvert}
\newcommand{\abss}[1]{\bigl\lvert #1 \bigr\rvert}
\newcommand{\norm}[1]{\left\lVert #1 \right\rVert}
\newcommand{\inner}[2]{\left\langle#1,#2\right\rangle}

\newcommand{\cc}[1]{\overline{#1}}
\newcommand{\floor}[1]{\lfloor#1 \rfloor}
\newcommand{\mb}[1]{\heavysymbol{#1}}

\newcommand{\defeq}{\mathrel{\mathop:}=}

\newcommand{\sA}{\cA}
\newcommand{\bsA}{\mb{\mathfrak{A}}}
\newcommand{\om}{\omega}
\newcommand{\bom}{\underline{\mb{\omega}}}
\newcommand{\ub}{\tau}

\title {\bf The Dirichlet problem for perturbed Stark operators in the half-line}

\author[1]{Julio H. Toloza%
		\thanks{\faEnvelopeO\, {julio.toloza@uns.edu.ar (corresponding author)}}%
			}
\author[2]{Alfredo Uribe%
		\thanks{\faEnvelopeO\, {alur@xanum.uam.mx}}%
			}
\affil[1]{Instituto de Matemática (INMABB)\\
		 Departamento de Matemática\\
		 Universidad Nacional del Sur (UNS) - CONICET\\
		 Bahía Blanca\\
		 Argentina}
\affil[2]{Departamento de Matemáticas\\
		 Universidad Autónoma Metropolitana\\
		 Av. San Rafael Atlixco 186\\
		 Col. Vicentina, Iztapalapa, C.P. 09340, México D.F.\\
		 Mexico}

\begin{document}

\date{}
\maketitle

\begin{abstract}
We consider the perturbed Stark operator
$H_q\varphi = -\varphi'' + x\varphi + q(x)\varphi$, $\varphi(0)=0$,
in $L^2(\R_+)$, where $q$ is a real function that belongs to
$\bsA_r =\left\{ q\in\sA_r\cap\text{AC}[0,\infty) : q'\in\sA_r\right\}$,
where $\sA_r = L^2_\R(\R_+,(1+x)^r dx)$ and $r>1$ is arbitrary but fixed.
Let $\left\{\lambda_n(q)\right\}_{n=1}^\infty$ and
$\left\{\kappa_n(q)\right\}_{n=1}^ \infty$ be the spectrum and associated set
of norming constants of $H_q$. Let $\{a_n\}_{n=1}^\infty$ be the zeros of the
Airy function of the first kind, and let $\omega_r:\N\to\R$ be defined by the rule
$\omega_r(n) = n^{-1/3}\log^{1/2}n$ if $r\in(1,2)$ and $\omega_r(n) = n^{-1/3}$
if $r\in[2,\infty)$. We prove that
$\lambda_n(q)
	= -a_n + \pi (-a_n)^{-1/2}\int_0^\infty \ai^2(x+a_n)q(x)dx
		+ O(n^{-1/3}\omega_r^2(n))$
and
$\kappa_n(q)
	= - 2\pi (-a_n)^{-1/2}\int_0^\infty \ai(x+a_n)\ai'(x+a_n)q(x)dx
		+ O(\omega_r^3(n))$,
uniformly on bounded subsets of $\bsA_r$. In order to obtain these asymptotic formulas,
we first show that $\lambda_n:\sA_r\to\R$ and $\kappa_n:\sA_r\to\R$ are real analytic maps.
\end{abstract}

\bigskip

\msc{
34E10,	
34L15,	
81Q05,	
81Q10	
}


\section{Introduction}

This paper is concerned with sharp asymptotics of the spectral data of perturbed
Stark operators, to wit,
\begin{equation*}
H_q \defeq -\frac{d^2}{dx^2} + x + q(x),
\quad
x\in \R_+ \defeq [0,\infty),
\end{equation*}
in $L^2(\R_+)$, adjoined with Dirichlet boundary condition at $x=0$.
We consider real-valued perturbations $q$ that belong to the Hilbert space
\begin{equation*}
\bsA_r \defeq \left\{ q\in\sA_r\cap\text{AC}[0,\infty) : q'\in\sA_r\right\},
\quad \norm{q}_{\bsA_r}^2 \defeq \norm{q}_{\sA_r}^2 + \norm{q'}_{\sA_r}^2,
\end{equation*}
where
\begin{equation*}
\sA_r \defeq L^2_\R(\R_+,(1+x)^r dx),\quad
	\norm{q}_{\sA_r} \defeq \norm{q}_{L^2(\R_+,(1+x)^r dx)},
\end{equation*}
and $r>1$ is arbitrary but fixed.

For perturbations in $\sA_r$, 
$H_q$ is a semi-bounded, self-adjoint operator. Moreover, it
has only simple, discrete spectrum, with a finite number of negative eigenvalues (if any).
Thus, $H_q$ is uniquely determined by the spectral data consisting of the set of eigenvalues
\begin{equation*}
\left\{\lambda_n(q) : \psi(q,\lambda_n(q),0) = 0\right\}_{n=1}^\infty
\end{equation*}
along with the set of corresponding (logarithmic) norming constants
\begin{equation*}
\left\{\kappa_n(q)
	= \log\left(-\frac{\psi'(q,\lambda_n(q),0)}
		{\dot{\psi}(q,\lambda_n(q),0)}\right)\right\}_{n=1}^\infty;
\end{equation*}
here $\psi(q,z,x)$ stands for the (unique up to a constant multiple) square integrable
solution to the eigenvalue problem
$-\varphi'' + [x + q(x)]\varphi = z\varphi$ ($z\in\C$),\footnote{The norming constants
are given by (minus) the residues of the Weyl function $m(z)$ ---which in this case
is a meromorphic Herglotz function--- at the eigenvalues. That is,
\[
e^{\kappa_n(q)}
	= -\lim_{\epsilon\to 0} i\epsilon\, m(q,\lambda_n(q) + i\epsilon)
	= \frac{\abs{\psi'(q,\lambda_n(q),0)}^2}{\norm{\psi(q,\lambda_n(q),\cdot)}_2^2}
	= -\frac{\psi'(q,\lambda_n(q),0)}{\dot{\psi}(q,\lambda_n(q),0)},
\]
where the last equality is consequence of the identity
$\partial_x(\psi\,\dot{\psi}' - \psi'\dot{\psi}) = -\psi^2$.
}
$\psi' = \partial_x\psi$ and $\dot{\psi} = \partial_z\psi$.
Needless to say, the same assertions holds true for perturbations in $\bsA_r$.

The spectral data of the unperturbed operator $H_0$ are, of course, easy to compute.
For in this case the eigenvalue problem $- \varphi'' + x \varphi = z \varphi$
has square integrable solution
\begin{equation*}
\label{eq:psi-0}
\psi_0(z,x) = \sqrt{\pi}\,\ai(x-z),
\end{equation*}
where $\ai(w)$ denotes the Airy function of the first kind (the factor $\sqrt{\pi}$ is
included for convenience). Hence,
\begin{equation*}
\lambda_n(0) = -a_n\quad\text{and}\quad \kappa_n(0) = 0
\end{equation*}
for every $n\in\N$, where $a_n$ denotes the $n$-th zero of the function $\ai(w)$.
It is well-known that
\begin{equation}
\label{eq:zeros-airy}
-a_n
	= \left(\tfrac32\pi\bigl(n-\tfrac14\bigr)\right)^{2/3} + O(n^{-4/3});
\end{equation}
see e.g. \cite[\S 9.9(iv)]{nist}.

The main results of the present paper are Proposition~\ref{lem:eingenvalue-is-real-analytic},
Theorem~\ref{thm:eigenvalues}, Proposition~\ref{lem:norming-constant-real-analytic} and
Theorem~\ref{thm:norming-constants}, the content of which can be summarized as follows.

\begin{thm}
For every $n\in\N$, $\lambda_n:\sA_r\to\R$ and $\kappa_n:\sA_r\to\R$
are real analytic maps.\footnote{Fréchet differentiability and related notions are
summarized at the beginning of Section~\ref{sec:frechet}.} Moreover, in terms of
\begin{equation}
\label{eq:omega_r}
\omega_r(n) \defeq	\begin{cases}
					n^{-1/3}\log^{1/2}n & \text{if } r\in(1,2),
					\\[1mm]
					n^{-1/3}			& \text{if } r\in[2,\infty),
					\end{cases}
\end{equation}
one has the following asymptotics:
\begin{equation*}
\lambda_n(q)
	= -a_n + \pi \frac{\int_0^\infty \ai^2(x+a_n)q(x)dx}{(-a_n)^{1/2}}
		+ O\bigl(n^{-1/3}\omega_r^2(n)\bigr)
\end{equation*}
and
\begin{equation*}
\kappa_n(q)
	= - 2\pi \frac{\int_0^\infty \ai(x+a_n)\ai'(x+a_n)q(x)dx}{(-a_n)^{1/2}}
		+ O\bigl(\omega_r^3(n)\bigr),
\end{equation*}
uniformly on bounded subsets of $\bsA_r$.
\end{thm}

We remark that it is sufficient to consider $\lambda_n$ and $\kappa_n$ as maps with
domain in $\sA_r$ in order to prove real analyticity.
However, the more restrictive
assumption $q\in\bsA_r$ is required to obtain the aforementioned asymptotic formulas.
Specifically, this requirement allows us to obtain sufficiently good estimates on
(partial) derivatives with
respect to the spectral parameter of solutions to the eigenvalue problem
associated to $H_q$ (see Section~\ref{sec:estimates}); these estimates in turn yield the
desired results.

The approach developed in this work is based on the (already classical) methods introduced
by Pöschel and Trubowitz in their treatment of the inverse Dirichlet problem
in a finite interval \cite{poeschel}. Some ideas are taken from
\cite{chelkak1,chelkak2,chelkak3}, where the inverse problem for the perturbed
harmonic oscillator is investigated (in the real line as well as in the half-line with
Dirichlet boundary condition).

We intend to use our results in the study of the associated isospectral problem; this
will be the subject of a subsequent paper.

There are few papers about the spectral analysis of Stark operators in
the half-line. In particular, some results on the direct and inverse spectral problem
are discussed in \cite{lk,mk}, where the authors use transformation operator
methods and the more restrictive assumption
$q\in C^{(1)}[0,\infty)\cap L^1(\R_+,x^4 dx)$, $q(x)=o(x)$ as $x\to\infty$.

By comparison, one-dimensional Stark operators on the real line have attracted much more
attention (for rather obvious reasons); see for
instance \cite{calogero,its,khanmamedov,katchalov,liu,sukhanov}.
As it is well-known, Stark operators on the real line are
characterized by the presence of  resonances; see \cite{korot1,korot2}
for some recent developments on this subject.

This paper is structured as follows: In Section~\ref{sec:prelim} we fix notation, discuss some
properties of the spaces $\sA_r$ and $\bsA_r$, and summarize some relevant information
concerning the unperturbed operator. In Section~\ref{sec:estimates} we obtain estimates
on $\psi(q,z,x)$ and other solutions to the eigenvalue problem associated to $H_q$.
Section~\ref{sec:frechet} is devoted to the Fréchet differentiability of $\psi(q,z,x)$.
The main results concerning the eigenvalues of $H_q$ are shown in
Section~\ref{sec:eigenvalues}. In Section~\ref{sec:auxiliary} we derive a number of
auxiliary results that are necessary in dealing with the norming constants. The main statements
about the norming constants are proven in Section~\ref{sec:norming-constants}. Finally,
some results related to Airy functions are presented in the Appendix.


\section{Preliminaries}
\label{sec:prelim}

The customary notation $u' \defeq \partial_x u$ and $\dot{u} \defeq \partial_z u$ is used
throughout this work.

An order relation of the form
\[
f(n) = O(g(n))
\]
always implicitly assume $n\in\N$ and (of course) $n\to\infty$.

The norm in $L^s(\R_+)$ is denoted $\norm{\cdot}_s$.

The complexification of a real Hilbert space $\cB$ is denoted $\cB^\C$. We use the
notation $\norm{\cdot}_\cB$ to represent the norm on either $\cB$ or $\cB^\C$; the same
convention holds for the inner product.

We observe that $\sA_r^\C\subset L^1(\R_+)$ if $r>1$ in which case
$\norm{q}_1\le (r-1)^{-1/2}\norm{q}_{\sA_r}$.
However, the function
\[
x\mapsto \frac{1}{(x+2)\log(x+2)}
\]
shows that this inclusion breaks down at $r=1$. Motivated by this observation ---and
a technical requirement that will be apparent in Section~\ref{sec:estimates}---,
\textit{throughout this paper the parameter $r$ is always assumed larger than 1}.
In passing, we note $\sA_{r_2}^\C\subset \sA_{r_1}^\C$ whenever $1\le r_1<r_2$.

\medskip

Let us define
\begin{equation*}
\label{eq:function-omega}
\om(q,z) \defeq \int_0^\infty\frac{\abs{q(x)}}{\sqrt{1+\abs{x-z}}}\,dx.
\end{equation*}

\begin{lemma}
\label{lem:about-omega}
Assume $q\in\sA_r^\C$. Then,
\[
\om(q,z)
	\le C \norm{q}_{\sA_r}\times \begin{dcases}
		\left(\frac{\log(2+\abs{z})}{2+\abs{z}}\right)^{1/2},& r\in(1,2),
		\\[1mm]
		\left(2+\abs{z}\right)^{-1/2},& r\in[2,\infty),
	\end{dcases}
\]
where (the lowest possible value of) $C >0$ only depends on $r$.
\end{lemma}
\begin{proof}
Let $n\defeq \floor{r}$. Then, Hölder inequality implies
\[
\int_0^\infty\frac{\abs{q(x)}}{\sqrt{1+\abs{x-z}}}\,dx
	\le \left(\int_0^\infty\frac{dx}{(1+\abs{x-z})(1+x)^n}\right)^{1/2}
		\left(\int_0^\infty\abs{q(x)}^2(1+x)^r dx\right)^{1/2}.
\]
Suppose $\abs{z}\ge 1$. Then,
\[
\int_0^\infty\frac{dx}{(1+\abs{x-z})(1+x)^n}
	\le \int_0^{\abs{z}}\frac{dx}{(1+\abs{z}-x)(1+x)^n}
		 + \int_{\abs{z}}^\infty\frac{dx}{(1-\abs{z}+x)(1+x)^n}.
\]
Using partial fraction decomposition\footnote{Specifically,
\[
\frac{1}{(a\mp x)(1+x)^n}
	= \frac{1}{(1\pm a)^n(a\mp x)} \pm \sum_{l=1}^n \frac{1}{(1\pm a)^{n+1-l}(1+x)^l}.
\]
}  one obtains
\begin{equation*}
\int_0^\infty\frac{dx}{(1+\abs{x-z})(1+x)}
	\le \left(\frac{2}{(2+\abs{z})} + \frac{1}{\abs{z}}\right)\log(1+\abs{z})
	\le C_1\frac{\log{(2+\abs{z}})}{2+\abs{z}}
\end{equation*}
if $n=1$, or
\begin{multline*}
\int_0^\infty\frac{dx}{(1+\abs{x-z})(1+x)^n}
\\[1mm]
	\le \left(\frac{2}{(2+\abs{z})^n} + \frac{1}{\abs{z}^n}\right)\log(1+\abs{z})
		+ \sum_{l=1}^{n-1}\frac{1-2(1+\abs{z})^{-l}}{l\abs{z}^{n-l}}
	\le C_n\frac{1}{2+\abs{z}}
\end{multline*}
if $n\ge 2$, where the value of $C_n>0$ (that is, its infimum) depends on $n$.
Now, set
\[
C = \max\left\{\sqrt{C_1/e}, D_n\right\} \quad\text{or}\quad
C = \max\left\{\sqrt{C_n/2}, D_n\right\}
\]
depending on the value of $n$, where
\[
D_n = \max_{z:\abs{z}\le1}\left(\int_0^\infty\frac{dx}
		{(1+\abs{x-z})(1+x)^n}\right)^{1/2}
\]
to complete the proof.
\end{proof}

We make extensive use of the following solutions to the unperturbed eigenvalue problem
$- \varphi'' + x \varphi = z \varphi$ ($x\in\R_+$, $z\in\C$),
\begin{align*}
\psi_0(z,x) &\defeq \sqrt{\pi}\ai(x-z),
	\\[1mm]
\theta_0(z,x) &\defeq \sqrt{\pi}\bi(x-z),
	\\[1mm]
s_0(z,x) &\defeq -\theta_0(z,0)\psi_0(z,x) + \psi_0(z,0)\theta_0(z,x),
	\\[1mm]
c_0(z,x) &\defeq \theta_0'(z,0)\psi_0(z,x) - \psi_0'(z,0)\theta_0(z,x),
\end{align*}
where $\bi(w)$ is the Airy function of the second kind (a few facts about Airy functions are
summarized in the Appendix). The solutions $\psi_0$ and
$\theta_0$ are normalized so that
\[
W(\psi_0(z),\theta_0(z)) \defeq \psi_0(z,x)\theta_0'(z,x)-\psi_0'(z,x)\theta_0(z,x) \equiv 1.
\]
As a result of Lemma~\ref{lem:g_Ag_C}, there exists a constant $C_0>0$ such that
\begin{align}
\abs{\psi_0(z,x)}
	&\le C_0\frac{g_A(x-z)}{\sigma{(x-z)}},\label{eq:vbe-psi0}
\\[1mm]
\abs{\theta_0(z,x)}
	&\le 2C_0\frac{g_B(x-z)}{\sigma{(x-z)}},\label{eq:vbe-theta0}
\end{align}
where $\sigma(w)\defeq 1 + \abs{w}^{1/4}$, $g_A(w) \defeq \exp(-\tfrac23\re w^{3/2})$
and $g_B(w) \defeq 1/g_A(w)$.\footnote{We find stylistically more pleasant to write $g_B$
instead of $g_A^{-1}$.} Moreover,
\begin{align*}
\abs{\psi_0'(z,x)}
	&\le C_0 \sigma{(x-z)}g_A(x-z),
\\[1mm]
\abs{\theta_0'(z,x)}
	&\le 2C_0 \sigma{(x-z)}g_B(x-z).
\end{align*}
It will be useful to bear in mind the identities
\[
\dot{\psi_0}(z,x) = - \psi_0'(z,x) \quad\text{and}\quad \dot{\theta_0}(z,x) = - \theta_0'(z,x).
\]
Regarding the function $g_A$ we mention the following result,
whose translation to $g_B$ is straightforward.

\begin{lemma}
\label{lem:stupid-lemma}
Given $z\in\C\setminus\R_+$, $x\mapsto g_A(x-z)$, $x\in\R_+$, is a monotonically decreasing
map such that $g_A(x-z)\to 0$ as $x\to\infty$. Given $\lambda\in\R_+$, then $g_A(x-\lambda) = 1$
if $x\in[0,\lambda]$ and monotonically decreases to zero if $x\in(\lambda,\infty)$.
\end{lemma}
\begin{proof}
Suppose $z\in\C_-$. A simple computation shows that, given $x\in\R$, there
exists a unique $\gamma\in(0,\pi)$ such that
\begin{equation*}
x-z= \frac{\abs{\im z}}{\sin\gamma}e^{i\gamma}.
\end{equation*}
Then,
\begin{equation*}
\re(x-z)^{3/2} = \abs{\im z}^{3/2}\frac{\cos\frac32\gamma}{(\sin\gamma)^{3/2}}.
\end{equation*}
The right hand side of the last equation is decreasing
as a function of $\gamma$. But the map $x\mapsto \gamma$ is also decreasing so
the map $x\mapsto \re(x-z)^{3/2}$ is increasing. This in turn implies the assertion.
Clearly, a similar reasoning works if $z\in\C_+$. The statement is obvious for $z\in\R$.
\end{proof}

The other pair of solutions, $s_0$ and $c_0$, obey the boundary conditions
\begin{equation*}
s_0(z,0) = c_0'(z,0) = 0,\quad s_0'(z,0) = c_0(z,0) = 1.
\end{equation*}
Let us define
\[
\ch(z,x)\defeq g_B(-z)g_A(x-z) + g_A(-z)g_B(x-z).
\]
Then, from the bounds on $\psi_0$, $\theta_0$ and their derivatives, one can readily obtain
\begin{align*}
\abs{s_0(z,x)}
	&\le 2C_0^2\frac{\ch(z,x)}{\sigma(z)\sigma(x-z)},
\\[1mm]
\abs{c_0(z,x)}
	&\le 2C_0^2 \frac{\sigma(z)}{\sigma(x-z)}\ch(z,x),
\end{align*}
as well as
\begin{align*}
\abs{s_0'(z,x)}
	&\le 2C_0^2\frac{\sigma(x-z)}{\sigma(z)}\ch(z,x),
\\[1mm]
\abs{c_0'(z,x)}
	&\le 2C_0^2 \sigma(z)\sigma(x-z)\ch(z,x).
\end{align*}
Let us also note the identity
\begin{equation*}
\dot{s}_0(z,x) = c_0(z,x) - s_0'(z,x),
\end{equation*}
which in turn implies
\begin{equation*}
\abs{\dot{s}_0(z,x)}
	\le 2C_0^2\left(\frac{\sigma(x-z)}{\sigma(z)} + \frac{\sigma(z)}{\sigma(x-z)}\right)\ch(z,x).
\end{equation*}

Finally, let us take a look at the Green function $J_0(z,x,y$) for the initial-value problem
related to the equation
\[
- \varphi'' + \left[x + q(x)\right]\varphi = z \varphi,\quad x\in\R_+,\quad z\in\C.
\]
Clearly, one has
\begin{align}
J_0(z,x,y)
	&= s_0(z,x)c_0(z,y) - c_0(z,x)s_0(z,y)\nonumber
\\[1mm]
	&= \theta_0(z,x)\psi_0(z,y) - \psi_0(z,x)\theta_0(z,y).\label{eq:green-function-0}
\end{align}
The second identity implies the inequalities
\begin{gather}
\abs{J_0(z,x,y)}
	\le 2C_0^2\frac{g_A(x-z)g_B(y-z) + g_B(x-z)g_A(y-z)}
		{\sigma(x-z)\sigma(y-z)},\label{eq:bound-J_0}
\\[1mm]
\abs{\partial_xJ_0(z,x,y)}
	\le 2C_0^2 \frac{\sigma(x-z)}{\sigma(y-z)}
		\left[g_A(x-z)g_B(y-z) + g_B(x-z)g_A(y-z)\right],\label{eq:partial-x-J_0}
\intertext{and}
\abs{\partial_yJ_0(z,x,y)}
	\le 2C_0^2 \frac{\sigma(y-z)}{\sigma(x-z)}
		\left[g_A(x-z)g_B(y-z) + g_B(x-z)g_A(y-z)\right].\label{eq:partial-y-J_0}
\end{gather}
An upper bound for the partial derivative $\partial_z J_0(z,x,y)$ follows immediately after
noticing that
\begin{equation}
\label{eq:dot-J_0}
\partial_z J_0(z,x,y) = - \partial_x J_0(z,x,y) - \partial_y J_0(z,x,y),
\end{equation}
as it is apparent from \eqref{eq:green-function-0}.


\section{Estimates on solutions and related functions}
\label{sec:estimates}

\subsection[One pair of solutions]
	{\protect\boldmath The solutions $\psi(q,z,x)$ and $\theta(q,z,x)$}

\begin{lemma}
\label{lem:basic-estimates}
Consider $q\in\sA_r^\C$.
Let $\psi(q,z,x)$ be the solution to the Volterra equation
\begin{equation}
\label{eq:volterra-psi}
\psi(q,z,x) = \psi_0(z,x) - \int_x^{\infty}J_0(z,x,y)\psi(q,z,y)q(y)dy\quad (z\in\C).
\end{equation}
The following statements hold true:
\begin{enumerate}[label={(\roman*)}]

\item \label{eq:psi} The function $\psi(q,z,\cdot)$ is the unique (up to a constant factor)
solution to the equation
\begin{equation}
\label{eq:stark-equation}
-\varphi'' + \left[x + q(x)\right]\varphi = z\varphi,
\end{equation}
that lies in $L^2(\R_+)$ for every $z\in\C$. This solution can be written as
\begin{equation*}
\psi(q,z,x) = \psi_0(z,x) + \varXi(q,z,x),
\end{equation*}
where
\begin{equation}
\label{eq:bound-psi}
\abs{\varXi(q,z,x)}
	\le C\om(q,z)e^{C\om(q,z)}
	\frac{g_A(x-z)}{\sigma(x-z)}.
\end{equation}
Also, $\psi(q,\cdot,x)$ is a real entire function for every $(q,x)\in\sA_r\times\R_+$.

\item \label{eq:psi-prime} Moreover,
\begin{equation*}
\psi'(q,z,x) = \psi_0'(z,x) + \varXi'(q,z,x),
\end{equation*}
where
\begin{equation}
\label{eq:bound-psi-prime}
\abs{\varXi'(q,z,x)}
	\le C\om(q,z)e^{C\om(q,z)}\sigma(x-z) g_A(x-z),
\end{equation}
and $\psi'(q,\cdot,x)$ is also a real entire function for every $(q,x)\in\sA_r\times\R_+$.

\item \label{eq:psi-dot} Finally,
\begin{equation}
\label{eq:psi-dot-crude}
\dot{\psi}(q,z,x) = - \psi_0'(z,x) + \dot{\varXi}(q,z,x),
\end{equation}
where
\begin{equation}
\label{eq:bound-psi-dot-crude}
\abs{\dot{\varXi}(q,z,x)}
	\le C \norm{q}_1 e^{C\norm{q}_1}\sigma(x-z) g_A(x-z).
\end{equation}
\end{enumerate}
\end{lemma}
\begin{proof}
In what follows we use the abbreviated notation $\psi(z,x) \defeq \psi(q,z,x)$.

\ref{eq:psi} A simple computation shows that a solution to \eqref{eq:volterra-psi}
is also a solution to \eqref{eq:stark-equation}.
Let $\psi_n(z,x)$ ($n\in\N$) be given by the recursive equation
\begin{equation*}
\psi_n(z,x) \defeq - \int_x^\infty J_0(z,x,y)\psi_{n-1}(z,y) q(y)dy.
\end{equation*}
Clearly, every such $\psi_n(\cdot,x)$ is an entire function. Using \eqref{eq:vbe-psi0},
\eqref{eq:bound-J_0} and Lemma~\ref{lem:stupid-lemma}, it is not difficult to
verify\footnote{Here we use the well-known identity
\begin{equation*}
\label{eq:identity-Poschel}
\int_x^\infty\int_{y_1}^\infty\cdots\int_{y_{n-1}}^\infty\prod_{l=1}^n
h(y_l)\,dy_1\cdots dy_n
	= \frac{1}{n!}\left[\int_x^\infty h(y) dy\right]^n.
\end{equation*}}
that
\begin{equation*}
\abs{\psi_n(z,x)}
\le \frac{4^n}{n!}C_0^{2n+1}\frac{g_A(x-z)}{\sigma(x-z)}
		\left(\int_x^{\infty}\frac{\abs{q(y)}}{\sigma(y-z)^2}dy\right)^n.
\end{equation*}
Since, in addition,
\[
\int_x^{\infty}\frac{\abs{q(y)}}{\sigma(y-z)^2}dy \le \om(q,z),
\]
we can obtain a solution to \eqref{eq:volterra-psi} as
\begin{equation}
\label{eq:series-for-psi}
\psi(z,x) = \sum_{n=0}^\infty \psi_n(z,x),
\end{equation}
where the convergence is uniform on bounded subsets of $\sA_r^\C\times\C\times\R_+$.
As a result, $\psi(q,\cdot,x)$ is an entire function, real entire whenever $q\in\sA_r$.
Moreover, \eqref{eq:bound-psi} implies that the solution so obtained
lies in $L^2(\R_+)$, hence it is unique (up to a constant multiple) because $H_q$ is in the
limit point case at $+\infty$.

The proof of \ref{eq:psi-prime} is analogous; for later use we note that
\begin{equation*}
\psi'(z,x) = \sum_{n=0}^\infty \psi_n'(z,x),
\end{equation*}
where
\begin{equation}
\label{eq:psi-prime-n}
\abs{\psi_n'(z,x)}
\le \frac{4^n}{n!}C_0^{2n+1} \sigma(x-z) g_A(x-z)
		\left(\int_x^{\infty}\frac{\abs{q(y)}}{\sigma(y-z)^2}dy\right)^n
\end{equation}
so the convergence is again uniform on bounded subsets of $\sA_r^\C\times\C\times\R_+$.

Finally, let us show \ref{eq:psi-dot}. Clearly, $\dot{\psi}(z,x)$ obeys the integral equation
\begin{equation}
\label{eq:dot-volterra}
\dot{\psi}(z,x) = \dot{\psi}_0(z,x)
	- \int_{x}^{\infty} \partial_z J_0(z,x,y) \psi(z,y) q(y) dy
	- \int_{x}^{\infty} J_0(z,x,y) \dot{\psi}(z,y) q(y) dy.
\end{equation}
Let $\rho_n(z,x)$ ($n\in\N$) be given by the recursive formula
\begin{equation*}
\rho_n(z,x)
	\defeq -\int_x^\infty \partial_z J_0(z,x,y)\psi_{n-1}(z,y)q(y)dy
		- \int_x^\infty J_0(z,x,y)\rho_{n-1}(z,y)q(y)dy,
\end{equation*}
along with $\rho_0(z,x)\defeq\dot{\psi_0}(z,x)$.
Resorting to \eqref{eq:vbe-psi0}, \eqref{eq:partial-x-J_0}, \eqref{eq:partial-y-J_0},
\eqref{eq:dot-J_0}, and Lemma~\ref{lem:stupid-lemma}, one can verify that
\begin{align*}
\abs{\rho_n(z,x)}
	&\le 4^nC_0^{2n+1}\sigma(x-z)g_A(x-z)\frac{1}{n!}
	\left(\int_x^\infty \frac{\abs{q(y)}}{\sigma(y-z)^2}dy\right)^n  \nn
\\[1mm]
	&\qquad + 2^{2n+1}C_0^{2n+1}\frac{g_A(x-z)}{\sigma(x-z)}\frac{1}{(n-1)!}
	\left(\int_x^\infty \abs{q(y)} dy\right)^{n} \nn
\\[1mm]
	&\le 4^nC_0^{2n+1}g_A(x-z)\left[\sigma(x-z)
		\frac{\om(q,z)^n}{n!}
		+ \frac{2}{\sigma(x-z)}\frac{\norm{q}_1^n}{(n-1)!}\right],
\end{align*}
from which \eqref{eq:bound-psi-dot-crude} follows.
\end{proof}

The next lemma provides and characterizes another solution, denoted $\theta(q,z,x)$,
to the eigenvalue problem \eqref{eq:stark-equation};
its proof is omitted since it is similar to the proof of Lemma~\ref{lem:basic-estimates}.
Lemma~\ref{lem:wronskian} then shows that this second solution is indeed linearly
independent of $\psi(q,z,x)$.

\begin{lemma}
\label{lem:basic-estimates-theta}
Assume $q\in\sA_r^\C$. Let $\theta(q,z,x)$ be the solution to the Volterra equation
\begin{equation}
\label{eq:volterra-theta}
\theta(q,z,x) = \theta_0(z,x) + \int_0^x J_0(z,x,y)\theta(q,z,y)q(y)dy \quad (z\in\C).
\end{equation}
The following statements hold true:
\begin{enumerate}[label={(\roman*)}]

\item \label{eq:theta} The function $\theta(z,\cdot)$ is a solution to equation
\eqref{eq:stark-equation} that can be written as
\begin{equation*}
\theta(q,z,x) = \theta_0(z,x) + \varGamma(q,z,x),
\end{equation*}
where
\begin{equation*}
\label{eq:bound-theta}
\abs{\varGamma(q,z,x)}
	\le C\om(q,z)e^{C\om(q,z)}\frac{g_B(x-z)}{\sigma(x-z)}.
\end{equation*}
Also, $\theta(q,\cdot,x)$ is a real entire function for every $(q,x)\in\sA_r\times\R_+$.

\item \label{eq:theta-prime} Moreover,
\begin{equation*}
\theta'(q,z,x) = \theta_0'(z,x) + \varGamma'(q,z,x),
\end{equation*}
where
\begin{equation*}
\label{eq:bound-theta-prime}
\abs{\varGamma'(q,z,x)}
	\le C\om(q,z)e^{C\om(q,z)}\sigma(x-z) g_B(x-z),
\end{equation*}
and $\theta'(q,\cdot,x)$ is a real entire function for every $(q,x)\in\sA_r\times\R_+$.

\item \label{eq:theta-dot} Finally,
\begin{equation*}
\dot{\theta}(q,z,x) = - \theta_0'(z,x) + \dot{\varGamma}(q,z,x),
\end{equation*}
where
\begin{equation*}
\abs{\dot{\varGamma}(q,z,x)}
	\le C \norm{q}_1 e^{C\norm{q}_1}\sigma(x-z) g_B(x-z).
\end{equation*}
\end{enumerate}
\end{lemma}

\begin{lemma}
\label{lem:wronskian}
Suppose $q\in\sA_r^\C$. Then $W(\psi(q,z),\theta(q,z))\equiv 1$.
\end{lemma}
\begin{proof}
As before, let us drop the argument $q$ for the sake of brevity. We have
\begin{equation*}
W(\psi(z),\theta(z)) = \psi(z,0)\theta'(z,0) - \psi'(z,0)\theta(z,0).
\end{equation*}
Some few computations involving \eqref{eq:volterra-psi}, \eqref{eq:volterra-theta} and their
derivatives yield
\begin{equation*}
W(\psi(z),\theta(z)) = 1 + \int_0^\infty \theta_0(z,y)\psi(z,y) q(y) dy.
\end{equation*}
However,
\begin{equation}
\label{eq:something-is-0}
\int_0^\infty \abs{\theta_0(z,y)}\abs{\psi(z,y)}\abs{q(y)} dy
	\le C\om(q,z) e^{C\om(q,z)}
\end{equation}
for some constant $C>0$. In view of Lemma~\ref{lem:about-omega}, \eqref{eq:something-is-0}
implies that $W(\psi(z),\theta(z))$ is a bounded, entire function so it is necessarily
constant. Clearly, this constant equals 1.
\end{proof}

\begin{remark}
\label{rem:useful-bounds}
Let us summarize the following consequences of Lemma~\ref{lem:basic-estimates} and
Lemma~\ref{lem:basic-estimates-theta}: There exists a ($r$-dependent) positive
constant $C_1$ such that, by defining
\[
\ub(q) \defeq C_1e^{C_1\norm{q}_{\sA_r}},
\]
one has
\begin{align*}
\abs{\psi(q,z,x)}
	&\le \ub(q)\frac{g_A(x-z)}{\sigma(x-z)}
	\\[1mm]
\abs{\psi'(q,z,x)}
	&\le \ub(q)\sigma(x-z) g_A(x-z)
	\\[1mm]
\abs{\dot{\psi}(q,z,x)}
	&\le \ub(q)\sigma(x-z) g_A(x-z)
	\\[1mm]
\abs{\theta(q,z,x)}
	&\le \ub(q)\frac{g_B(x-z)}{\sigma(x-z)}
	\\[1mm]
\abs{\theta'(q,z,x)}
	&\le \ub(q)\sigma(x-z) g_B(x-z)
	\\[1mm]
\abs{\dot{\theta}(q,z,x)}
	&\le \ub(q)\sigma(x-z) g_B(x-z),
\end{align*}
for all $q\in\sA_r^\C$. These inequalities will be used in Section~\ref{sec:frechet}.
\end{remark}

In order to deal with the asymptotics of the norming constants we need sharper estimates
of $\dot{\psi}(q,z,x)$ and its (partial) derivatives.
This requirement leads to the introduction of the space $\bsA_r^\C$ and, with it, the
function
\[
\bom(q,z) \defeq \om(q,z) + \om(q',z)\quad (q\in\bsA_r^\C).
\]

\begin{lemma}
\label{lem:dot-psi-refined}
Suppose $q\in\bsA_r^\C$.
Then, the error term $\dot{\varXi}(q,z,x)$ in \eqref{eq:psi-dot-crude} obeys
\begin{equation}
\label{eq:bound-psi-dot}
\abs{\dot{\varXi}(q,z,x)}
	\le C\bom(q,z)e^{C\bom(q,z)} \sigma(x-z) g_A(x-z).
\end{equation}
\end{lemma}
\begin{proof}
By using \eqref{eq:dot-J_0} followed by an integration by parts, along with the fact
that $q\in\bsA^\C_r$ implies that $q$ is bounded, we transform \eqref{eq:dot-volterra}
into
\begin{multline*}
\dot{\psi}(z,x) + \psi'(z,x)
	= - \int_{x}^{\infty} J_0(z,x,y) \psi(z,y) q'(y) dy
\\
		- \int_{x}^{\infty} J_0(z,x,y) \left(\dot{\psi}(z,y) + \psi'(z,y)\right) q(y) dy.
\end{multline*}
Let $\beta_n(z,x)$ ($n\in\N$) be given by the recursive rule
\begin{equation}
\label{eq:beta}
\beta_n(z,x) \defeq - \int_{x}^{\infty} J_0(z,x,y) \psi_{n-1}(z,y) q'(y) dy
		- \int_{x}^{\infty} J_0(z,x,y) \beta_{n-1}(z,y) q(y) dy,
\end{equation}
where $\beta_{0}(z,x)= 0$. Then,
\begin{equation}
\label{eq:beta-n}
\abs{\beta_n(z,x)}
	\le 2^{n-1} \frac{4^n}{n!}C_0^{2n+1}\frac{g_A(x-z)}{\sigma(x-z)}
			\left(\int_x^{\infty}\frac{\abs{q(y)}+\abs{q'(y)}}
			{\sigma(y-z)^2}dy\right)^n.
\end{equation}
It follows that
\[
\dot{\psi}(z,x) + \psi'(z,x) = \sum_{n=1}^\infty \beta_n(z,x),
\]
where the convergence is uniform on bounded subsets of $\bsA_r^\C\times\C\times\R_+$, and
\[
\abs{\dot{\psi}(z,x) + \psi'(z,x)}
	\le 4 C_0^3\bom(q,z)e^{8C_0^ 2\bom(q,z)}\frac{g_A(x-z)}{\sigma(x-z)}.
\]
Inequality \eqref{eq:bound-psi-dot} follows from \eqref{eq:bound-psi-prime} and the
last estimate.
\end{proof}

\begin{remark}
\label{rem:dot-psi}
For later use, we note that
\[
\dot{\psi}(z,x)
	= \sum_{n=0}^\infty \dot{\psi}_n(z,x)
	= - \psi_0'(z,x) + \sum_{n=1}^\infty \left[\beta_n(z,x) - \psi_n'(z,x)\right],
\]
where
\[
\abs{\dot{\psi}_n(z,x)}
	\le \frac{4^n}{n!}C_0^{2n+1}g_A(x-z)\left[\sigma(x-z) +
		\frac{2^{n-1}}{\sigma(x-z)}\right]
		\left(\int_x^{\infty}\frac{\abs{q(y)}+\abs{q'(y)}}
				{\sigma(y-z)^2}dy\right)^n
\]
for all $n\ge 1$; this follows from \eqref{eq:psi-prime-n} and \eqref{eq:beta-n}.
\end{remark}

\begin{remark}
\label{rem:more-about-psi}
Some of the estimates concerning the norming constants will require a more refined
decomposition. Namely, assuming $q\in\sA_r^\C$,
\begin{equation*}
\varXi(q,z,x) = \psi_1(q,z,x) + \varXi^{(2)}(q,z,x),
\end{equation*}
where
\begin{equation*}
\psi_1(q,z,x) \defeq - \int_x^\infty J_0(z,x,y)\psi_0(z,y)q(y)dy,
\end{equation*}
and
\begin{equation*}
\abs{\varXi^{(2)}(q,z,x)}
	\le C\om^2(q,z)e^{C\om(q,z)}
	\frac{g_A(x-z)}{\sigma(x-z)}.
\end{equation*}
Also,
\begin{equation*}
\varXi'(q,z,x) =  \psi_1'(q,z,x) + {\varXi^{(2)\prime}}(q,z,x),
\end{equation*}
where
\begin{equation*}
\abs{{\varXi^{(2)\prime}}(q,z,x)}
	\le C\om^2(q,z)e^{C\om(q,z)}\sigma(x-z)g_A(x-z).
\end{equation*}
Moreover, if $q\in\bsA_r^\C$,
\begin{equation*}
\dot{\varXi}(q,z,x) = - \psi_1'(q,z,x) + \dot{\psi}_1^\text{res}(q,z,x)
					+ \dot{\varXi}^{(2)}(q,z,x),
\end{equation*}
where
\begin{equation*}
\dot{\psi}_1^\text{res}(q,z,x)
	\defeq -\int_x^\infty J_0(z,x,y)\psi_0(z,y) q'(y) dy,
\end{equation*}
and
\begin{equation*}
\abs{\dot{\varXi}^{(2)}(q,z,x)}
	\le C\bom^2(q,z) e^{C\bom(q,z)}\sigma(x-z)g_A(x-z).\qedhere
\end{equation*}
\end{remark}

\begin{lemma}
\label{lem:even-more-about-psi}
Suppose $q\in\bsA_r^\C$. Denote
$\norm{q}_{\underline{\mb{1}}} \defeq \norm{q}_1 + \norm{q'}_1$. Then,
\[
\dot{\psi}'(q,z,x) = -(x-z)\psi_0(z,x) + \dot{\varXi}'(q,z,x),
\]
where
\begin{equation}
\label{eq:K}
\abs{\dot{\varXi}'(q,z,x)}
	\le Ce^{C\bom(q,z)}\left[\bigl(1+\abs{x-z}^{3/4}\bigr)\bom(q,z)
		+ \sigma(x-z)\norm{q}_1\right] g_A(x-z).
\end{equation}
Also,
\begin{equation}
\label{eq:dot-beta}
\abs{\dot{\psi}'(q,z,x) + \ddot{\psi}(q,z,x)}
	\le Ce^{C\bom(q,z)}
			\left[\sigma(x-z)\bom(q,z)
			+ \frac{\norm{q}_{\underline{\mb{1}}}}{\sigma(x-z)}\right] g_A(x-z).
\end{equation}
As a consequence,
\begin{equation}
\label{eq:double-dot-Xi}
\abs{\ddot{\varXi}(q,z,x)}
	\le Ce^{C\bom(q,z)}
			\left[\bigl(1+\abs{x-z}^{3/4}\bigr)\bom(q,z)
			+ \sigma(x-z)\norm{q}_{\underline{\mb{1}}}\right] g_A(x-z).
\end{equation}
\end{lemma}
\begin{proof}
Let us abbreviate $\psi(z,x) = \psi(q,z,x)$, $\varXi(z,x) = \varXi(q,z,x)$ and so on.
We have the integral equation
\begin{multline*}
\dot{\psi}'(z,x) = -(x-z)\psi_0(z,x)
	+ (x-z)\int_x^\infty J_0(z,x,y)\psi(z,y)q(y)dy
\\
	+ \int_x^\infty \partial_y\partial_x J_0(z,x,y)\psi(z,y)q(y)dy
	- \int_x^\infty \partial_x J_0(z,x,y)\dot{\psi}(z,y)q(y)dy.
\end{multline*}
Define $\dot{\psi}'_n(z,x)$ ($n\in\N$) by means of the equation
\begin{multline*}
\dot{\psi}'_n(z,x) \defeq (x-z)\int_x^\infty J_0(z,x,y)\psi_{n-1}(z,y)q(y)dy
	\\
		+ \int_x^\infty \partial_y\partial_x J_0(z,x,y)\psi_{n-1}(z,y)q(y)dy
		- \int_x^\infty \partial_x J_0(z,x,y)\dot{\psi}_{n-1}(z,y)q(y)dy,
\end{multline*}
where $\psi_n(z,x)$ is defined in the proof of Lemma~\ref{lem:basic-estimates}
and $\dot{\psi}_n(z,x)$ is defined in Remark~\ref{rem:dot-psi}. It follows that
\begin{multline*}
\abs{\dot{\psi}'_n(z,x)}
	\le 2\frac{4^n}{(n-1)!}C_0^{2n+1}\norm{q}_1\sigma(x-z)
		\bom(q,z)^{n-1} g_A(x-z)
\\
	+ \frac{4^n}{n!}C_0^{2n+1}
		\left[2^{n-1}\sigma(x-z) + \frac{\abs{x-z}}{\sigma(x-z)}\right]
		\bom(q,z)^n g_A(x-z).
\end{multline*}
Therefore,
\[
\dot{\varXi}'(z,x) = \sum_{n=1}^\infty \dot{\psi}'_n(z,x),
\]
where the convergence is uniform on bounded subsets of $\sA_r^\C\times\C\times\R_+$;
this in turn implies \eqref{eq:K}.

In accordance with the notation introduced in the proof of Lemma~\ref{lem:basic-estimates},
let us write $\dot{\beta}(z,x) \defeq \dot{\psi}'(z,x) + \ddot{\psi}(z,x)$. Then, \eqref{eq:beta}
implies the identity
\begin{multline*}
\dot{\beta}(z,x)
	= \int_{x}^{\infty} \partial_x J_0(z,x,y) \psi(z,y) q'(y) dy
		+ \int_{x}^{\infty} \partial_y J_0(z,x,y) \psi(z,y) q'(y) dy
		\\
		- \int_{x}^{\infty} J_0(z,x,y) \dot{\psi}(z,y) q'(y) dy
		+ \int_{x}^{\infty} \partial_x J_0(z,x,y) \beta(z,y) q(y) dy
		\\
		+ \int_{x}^{\infty} \partial_y J_0(z,x,y) \beta(z,y) q(y) dy
		- \int_{x}^{\infty} J_0(z,x,y) \dot{\beta}(z,y) q(y) dy.
\end{multline*}
We shall see that
\begin{equation}
\label{eq:series-for-beta}
\dot{\beta}(z,x) = \sum_{n=1}^\infty \dot{\beta}_n(z,x),
\end{equation}
where $\dot{\beta}_n(z,x)$ are given by the recursive equation
\begin{multline*}
\dot{\beta}_n(z,x)
	= \int_{x}^{\infty} \partial_x J_0(z,x,y) \psi_{n-1}(z,y) q'(y) dy
		+ \int_{x}^{\infty} \partial_y J_0(z,x,y) \psi_{n-1}(z,y) q'(y) dy
		\\
		- \int_{x}^{\infty} J_0(z,x,y) \dot{\psi}_{n-1}(z,y) q'(y) dy
		+ \int_{x}^{\infty} \partial_x J_0(z,x,y) \beta_{n-1}(z,y) q(y) dy
		\\
		+ \int_{x}^{\infty} \partial_y J_0(z,x,y) \beta_{n-1}(z,y) q(y) dy
		- \int_{x}^{\infty} J_0(z,x,y) \dot{\beta}_{n-1}(z,y) q(y) dy.
\end{multline*}
A not-so-painful computation shows that
\begin{multline*}
\abs{\dot{\beta}_n(z,x)}
	\le 2^{n+1}\frac{4^n}{(n-1)!}C_0^{2n+1}\frac{\norm{q}_{\underline{\mb{1}}}}{\sigma(x-z)}
		\bom(q,z)^{n-1} g_A(x-z)
\\
	+ 2^{n-1}\frac{4^n}{n!}C_0^{2n+1}
		\left[\sigma(x-z) + \frac{1}{\sigma(x-z)}\right]
			\bom(q,z)^n g_A(x-z),
\end{multline*}
which indeed implies the uniform convergence of \eqref{eq:series-for-beta} on bounded subsets
of $\sA_r^\C\times\C\times\R_+$, hence the estimate \eqref{eq:dot-beta}.

Finally, since
\[
\ddot{\psi}(q,z,x) = (x-z)\psi_0(z,x) + \ddot{\varXi}(q,z,x),
\]
we have
\[
\dot{\psi}'(q,z,x) + \ddot{\psi}(q,z,x) = \dot{\varXi}'(q,z,x) + \ddot{\varXi}(q,z,x),
\]
from which \eqref{eq:double-dot-Xi} follows.
\end{proof}

\subsection[The fundamental pair]
	{\protect\boldmath The fundamental pair $s(q,z,x)$ and $c(q,z,x)$}

The proofs of the following two statements (Lemmas~\ref{lem:about-s} and \ref{lem:about-c})
are based on arguments similar to those of Lemmas~\ref{lem:basic-estimates} and
\ref{lem:dot-psi-refined}; for the sake of brevity they are omitted.

\begin{lemma}
\label{lem:about-s}
Consider $q\in\sA_r^\C$.
Let $s(q,z,x)$ be the solution to the integral equation
\begin{equation*}
s(q,z,x) = s_0(z,x) + \int_0^x J_0(z,x,y)s(q,z,y)q(y) dy\quad (z\in\C).
\end{equation*}
Then, the following assertion hold true:
\begin{enumerate}[label={(\roman*)}]
\item\label{it:s}
	The function $s(q,z,x)$ is the solution to the eigenvalue equation
	\eqref{eq:stark-equation} that obey the boundary conditions
	\begin{equation*}
	s(q,z,0) = 0,\quad s'(q,z,0) = 1.
	\end{equation*}
	Also,
	\begin{equation*}
	s(q,z,x) = s_0(z,x) + \varUpsilon_s(q,z,x),
	\end{equation*}
	where
	\begin{equation*}
	\abs{\varUpsilon_s(q,z,x)}
		\le C\om(q,z)e^{C\om(q,z)}
			\frac{\ch(z,x)}{\sigma(z)\sigma(x-z)}.
	\end{equation*}

\item\label{it:s-prime}
	Moreover,
	\begin{equation*}
	s'(q,z,x) = s_0'(z,x) + \varUpsilon_s'(q,z,x),
	\end{equation*}
	where
	\begin{equation*}
	\abs{\varUpsilon_s'(q,z,x)}
		\le C\om(q,z)e^{C\om(q,z)}
			\frac{\sigma(x-z)}{\sigma(z)}\ch(z,x).
	\end{equation*}
\end{enumerate}
Besides, $s(q,\cdot,x)$ and $s'(q,\cdot,x)$ are real entire functions for every
$(q,x)\in\sA_r\times\R_+$.
\begin{enumerate}[resume*]
\item\label{it:s-dot}
	Assuming $q\in\bsA_r^\C$,
	\begin{equation*}
	\dot{s}(q,z,x) = c_0(z,x) - s_0'(z,x) + \dot{\varUpsilon}_s(q,z,x),
	\end{equation*}
	where
	\begin{equation*}
	\abs{\dot{\varUpsilon}_s(q,z,x)}
		\le C\bom(q,z)e^{C\bom(q,z)}
			\left(\frac{\sigma(x-z)}{\sigma(z)}+\frac{\sigma(z)}{\sigma(x-z)}\right)\ch(z,x).
	\end{equation*}
\end{enumerate}
\end{lemma}

\begin{lemma}
\label{lem:about-c}
Consider $q\in\sA_r^\C$.
Let $c(q,z,x)$ be the solution to the integral equation
\begin{equation*}
c(q,z,x) = c_0(z,x) + \int_0^x J_0(z,x,y)c(q,z,y)q(y) dy\quad (z\in\C).
\end{equation*}
Then,
	$c(q,z,x)$ is the entire solution (real entire whenever $q\in\sA_r$)
	to the eigenvalue equation
	\eqref{eq:stark-equation} that obey the boundary conditions
	\begin{equation*}
	c(q,z,0) = 1,\quad c'(q,z,0) = 0.
	\end{equation*}
	Also,
	\begin{equation*}
	c(q,z,x) = c_0(z,x) + \varUpsilon_c(q,z,x),
	\end{equation*}
	where
	\begin{equation*}
	\abs{\varUpsilon_c(q,z,x)}
		\le C\om(q,z)e^{C\om(q,z)}
			\frac{\sigma(z)}{\sigma(x-z)}\ch(z,x).
	\end{equation*}
\end{lemma}

\begin{remark}
\label{rem:even-more-about-s}
Later we shall make use of the decomposition
\begin{equation*}
\varUpsilon_s(q,z,x) = s_1(q,z,x) + \varUpsilon^{(2)}_s(q,z,x),
\end{equation*}
where
\begin{equation*}
s_1(q,z,x) = \int_0^x J_0(z,x,y)s_0(z,y)q(y) dy
\end{equation*}
and
\begin{equation*}
\abs{\varUpsilon^{(2)}_s(q,z,x)}
	\le C\om^2(q,z)e^{C\om(q,z)}
		\frac{\ch(z,x)}{\sigma(z)\sigma(x-z)},
\end{equation*}
of course under the assumption $q\in\sA_r^\C$. If moreover $q\in\bsA_r^\C$,
we also have
\begin{equation*}
\dot{\varUpsilon}_s(q,z,x) = c_1(q,z,x) - s'_1(q,z,x) + \dot{s}_1^\text{res}(q,z,x)
	+ \dot{\varUpsilon}^{(2)}_s(q,z,x),
\end{equation*}
where
\begin{gather*}
c_1(q,z,x)
	= \int_0^x J_0(z,x,y)c_0(z,y) q(y) dy,
\\[1mm]
\dot{s}_1^\text{res}(q,z,x)
	= \int_0^x J_0(z,x,y)s_0(z,y) q'(y) dy,
\end{gather*}
and
\begin{equation*}
\abs{\dot{\varUpsilon}^{(2)}_s(q,z,x)}
	\le C\bom^2(q,z)e^{C\bom(q,z)}
	\left(\frac{\sigma(x-z)}{\sigma(z)}+\frac{\sigma(z)}{\sigma(x-z)}\right)\ch(z,x).
\end{equation*}
Finally,
\begin{equation*}
\varUpsilon_c(q,z,x) = c_1(q,z,x) + \varUpsilon^{(2)}_c(q,z,x),
\end{equation*}
where
\begin{equation*}
c_1(q,z,x) = \int_0^x J_0(z,x,y)c_0(z,y)q(y) dy
\end{equation*}
\begin{equation*}
\abs{\varUpsilon^{(2)}_c(q,z,x)}
		\le C\om^2(q,z)e^{C\om(q,z)}
			\frac{\sigma(z)}{\sigma(x-z)}\ch(z,x).\qedhere
\end{equation*}
\end{remark}


\section[Fréchet differentiability]
		{\protect\boldmath Fréchet differentiability of $\psi(q,z,x)$ and its derivatives}
\label{sec:frechet}

To start with, let us recall the following definitions, particularized to the case in hand:
Let $\cB$ be a Hilbert space over $\K$, and let $\cU\subset\cB$ be open. A map $f:\cU\to\K$
is (Fréchet) differentiable at $q\in\cU$ if there exists a linear functional $d_qf:\cB\to\K$
such that
\[
\lim_{v\to 0}\frac{\abs{f(q+v) - f(q) - d_qf(v)}}{\norm{v}_\cB} = 0.
\]
The map $f$ is continuously differentiable on $\cU$ if it is differentiable at every point
in $\cU$ and the resulting map $df:\cU\to L(\cB,\K)$ is continuous. If $\cB$ is a complex
Hilbert space, then $f$ is analytic on an open subset $\cU$ of $\cB$ if it is
continuously differentiable there.
Now, let $\cB^\C$ be the complexification of a real Hilbert space $\cB$ and assume
$f:\cV\to\C$ differentiable at $q\in\cV$ (an open subset of $\cB^\C$). Then
the gradient of $f$ at $q$ is the (unique) element
$\partial f/\partial q\in\cB^\C$ such that
\[
d_qf(v) = \inner{\cc{\frac{\partial f}{\partial q}}}{v}_{\cB}
\]
for all $v\in\cB^\C$.
Finally, consider a real Hilbert space $\cB$ and let
$\cU\subset\cB$ be open. Then $f:\cU\to\R$ is real analytic on $\cU$ if for every $q\in\cU$
there exists $\cV_q\subset\cB^\C$ open and an analytic map $h_q:\cV_q\to\C$ such that
$f(v) = h_q(v)$ for all $v\in\cU\cap\cV_q$ (assumed non-empty).

\medskip

The Green function for the initial-value problem related to the equation
\[
-\varphi'' + \left[x + q(x) + v(x)\right]\varphi = z\varphi,\quad x\in\R_+,\quad z\in\C,
\]
is given by
\begin{align*}
J(q,z,x,y)
	&= s(q,z,x) c(q,z,y) - c(q,z,x) s(q,z,y)
	\\[1mm]
	&= \theta(q,z,x)\psi(q,z,y) - \psi(q,z,x)\theta(q,z,y),
\end{align*}
where the second identity follows a computation involving Lemma~\ref{lem:wronskian}.
Recalling Remark~\ref{rem:useful-bounds}, the second identity above implies
\begin{gather}
\abs{J(q,z,x,y)}
	\le \ub(q)^2
		\frac{g_A(x-z)g_B(y-z) + g_B(x-z)g_A(y-z)}
		{\sigma(x-z)\sigma(y-z)}
		\label{eq:bound-J},
		\\[1mm]
\abs{\partial_xJ(q,z,x,y)}
	\le \ub(q)^2 \frac{\sigma(x-z)}{\sigma(y-z)}
		\left[g_A(x-z)g_B(y-z) + g_B(x-z)g_A(y-z)\right]
		\label{eq:bound-J-prime}
\end{gather}
and
\begin{multline}
\label{eq:bound-J-dot}
\abs{\partial_z J(q,z,x,y)}
	\le \ub(q)^2 \left(\frac{\sigma(x-z)}{\sigma(y-z)}
			+ \frac{\sigma(y-z)}{\sigma(x-z)}\right)
	\\
		\times\left[g_A(x-z)g_B(y-z) + g_B(x-z)g_A(y-z)\right],
\end{multline}
for all $q\in\sA_r^\C$.

\begin{lemma}
$\psi(\cdot,z,x)$ and $\psi'(\cdot,z,x)$ are analytic maps
from $\sA_r^\C$ to $\C$. Moreover, their gradients are given by
\begin{align}
\frac{\partial\psi}{\partial q(y)}(q,z,x)
	&= - \cc{J(q,z,x,y)\psi(q,z,y)}\chi_{[x,\infty)}(y)(1+y)^{-r},
	\label{eq:gradient-psi-complex}
\intertext{and}
\frac{\partial\psi'}{\partial q(y)}(q,z,x)
	&= - \cc{\partial_x J(q,z,x,y)\psi(q,z,y)}\chi_{[x,\infty)}(y)(1+y)^{-r}.
	\nn
\end{align}
\end{lemma}

\begin{proof}
Fix $q\in\sA_r^\C$. Since $\psi(q+v,z,x)$ is solution to the Volterra equation
\begin{equation}
\label{eq:psi-q-plus-v}
\psi(q+v,z,x) = \psi(q,z,x) - \int_x^{\infty}J(q,z,x,y)\psi(q+v,z,y)v(y)dy,
\end{equation}
we begin by considering the functions $\psi_n(q+v,z,x)$ ($n\in\N$) defined by the
recursive formula
\begin{equation}
\label{eq:recursive-formula-gradient-psi}
	\psi_n(q+v,z,x) \defeq - \int_x^\infty J(q,z,x,y)v(y)\psi_{n-1}(q+v,z,y)dy,
\end{equation}
along with the initial condition $\psi_0(q+v,z,x) \defeq \psi(q,z,x)$.
Because of \eqref{eq:bound-J}, \eqref{eq:recursive-formula-gradient-psi} implies
\begin{equation}
\label{eq:inductive-hyp}
\abs{\psi_n(q+v,z,x)}
	\le \frac{2^n}{n!} \ub(q)^{2n+1}
		\frac{g_A(x-z)}{\sigma(x-z)}\left(\int_x^{\infty}\frac{\abs{v(y)}}
		{\sigma(y-z)^2}dy\right)^n.
\end{equation}
Thus, we obtain $\psi(q+v,z,x)$ by uniform convergence on bounded subsets of
$\sA_r^\C\times\C\times\R_+$ of the series
\begin{equation*}
\sum_{n=0}^\infty \psi_n(q+v,z,x).
\end{equation*}
Also, \eqref{eq:psi-q-plus-v} yields
\begin{equation*}
\psi(q+v,z,x) = \psi(q,z,x) - \int_x^{\infty}J(q,z,x,y)\psi(q,z,y)v(y)dy
	+ \varXi^{(2)}(q+v,z,x),
\end{equation*}
where
\begin{equation*}
\varXi^{(2)}(q+v,z,x) \defeq \sum_{n=2}^\infty \psi_n(q+v,z,x).
\end{equation*}
As a consequence of \eqref{eq:inductive-hyp},
\begin{equation*}
\abs{\varXi^{(2)}(q+v,z,x)}
	\le 4 \ub(q)^5 \om^2(v,z) e^{2\ub(q)^2\om(v,z)}
	\frac{g_A(x-z)}{\sigma(x-z)}.
\end{equation*}
Therefore, in view of Lemma~\ref{lem:about-omega},
\begin{equation*}
\frac{\abs{\varXi^{(2)}(q+v,z,x)}}{\norm{v}_{\sA_r}} \to 0, \quad v \to 0,
\end{equation*}
uniformly on bounded subsets of $\sA_r^\C$. But this implies that
$\psi(\cdot,z,x):\sA_r^\C\to\C$ is continuously differentiable, and clearly
\[
\inner{\cc{\frac{\partial\psi}{\partial q}(q,z,x)}}{v}_{\sA_r}
	= - \int_x^{\infty}J(q,z,x,y)\psi(q,z,y)v(y)dy,
\]
thus proving \eqref{eq:gradient-psi-complex}.
The proof concerning $\psi'(\cdot,z,x)$ is omitted since it goes along a similar
argument, in this case based on \eqref{eq:bound-J-prime}.
\end{proof}

\begin{lemma}
$\dot{\psi}(\cdot,z,x)$ is an analytic map from $\sA_r^\C$ to $\C$, whose gradient is
\[
\frac{\partial\dot{\psi}}{\partial q(y)}(q,z,x)
	= - \cc{\left(\partial_z J(q,z,x,y)\psi(q,z,y)
		+ J(q,z,x,y)\dot{\psi}(q,z,y)\right)}\chi_{[x,\infty)}(y)(1+y)^{-r}.
\]
\end{lemma}
\begin{proof}
Again, fix $q\in\sA_r^\C$. It follows from
\eqref{eq:psi-q-plus-v} that $\dot{\psi}(q+v,z,x)$ satisfies the integral
equation
\begin{multline*}
\dot{\psi}(q+v,z,x) = \dot{\psi}(q,z,x)
	- \int_x^\infty \partial_z J(q,z,x,y) \psi(q+v,z,y) v(y) dy
\\[1mm]
	- \int_x^\infty J(q,z,x,y) \dot{\psi}(q+v,z,y) v(y) dy.
\end{multline*}
Let $\zeta_n(q+v,z,x)$ ($n\in\N$) be given by the recursive formula
\begin{multline*}
\zeta_n(q+v,z,x) \defeq
	-\int_x^\infty \partial_z J(q,z,x,y)\psi_{n-1}(q+v,z,y)v(y)dy
\\[1mm]
	-\int_x^\infty J(q,z,x,y)\zeta_{n-1}(q+v,z,y)v(y)dy,
\end{multline*}
along with $\zeta_0(q+v,z,x)\defeq\dot{\psi}(q,z,x)$,
where $\psi_n(q+v,z,y)$ has been defined in \eqref{eq:recursive-formula-gradient-psi}.
Then, an induction argument involving \eqref{eq:bound-J} and \eqref{eq:bound-J-dot}
shows that
\begin{multline*}
\abs{\zeta_n(q+v,z,x)}
	\le \frac{2^n}{n!}\tau(q)^{2n+1}\sigma(x-z)g_A(x-z)
		\left(\int_x^{\infty}\frac{\abs{v(y)}}{\sigma(y-z)^2}dy\right)^n
	\\[1mm]
	+ \frac{2^{n+1}}{(n-1)!}\tau(q)^{2n+1}\frac{g_A(x-z)}{\sigma(x-z)}
				\left(\int_x^{\infty}\abs{v(y)}dy\right)^n,
\end{multline*}
hence
\begin{align}
\abs{\zeta_n(q+v,z,x)}
	&\le \frac{2^n}{n!}\tau(q)^{2n+1}\left[\sigma(x-z)\om(v,z)^n
		+ \frac{2 n}{\sigma(x-z)}\norm{v}_{\sA_r}^n\right]g_A(x-z)\nn
	\\[1mm]
	&\le C\frac{2^n}{(n-1)!}\tau(q)^{2n+1}\norm{v}_{\sA_r}^n
		\sigma(x-z)g_A(x-z).\label{eq:error-dot}
\end{align}
Therefore,
\begin{multline*}
\dot{\psi}(q+v,z,x)
	= \dot{\psi}(q,z,x)
		- \int_x^\infty \left(\partial_z J(q,z,x,y)\psi(q,z,y)
		+ J(q,z,x,y) \dot{\psi}(q,z,y)\right)v(y) dy
	\\[1mm]
	+ \dot{\varXi}^{(2)}(q+v,z,x),
\end{multline*}
where
\[
\dot{\varXi}^{(2)}(q+v,z,x) = \sum_{n=2}^\infty\zeta_n(q+v,z,x).
\]
As a result of \eqref{eq:error-dot},
\begin{equation*}
\frac{\abs{\dot{\varXi}^{(2)}(q+v,z,x)}}{\norm{v}_{\sA_r}} \to 0, \quad v \to 0,
\end{equation*}
uniformly on bounded subsets of $\sA_r^\C$. Thus, the assertion is proven.
\end{proof}

\begin{remark}
\label{rem:everything-has-gradient}
In particular, under the assumption $(q,\lambda)\in\sA_r\times\R$,
\begin{align}
\frac{\partial\psi}{\partial q(y)}(q,\lambda,0)
	&=  s(q,\lambda,y)\psi(q,\lambda,y)(1+y)^{-r},\label{eq:gradient-psi}
\\[1mm]
\frac{\partial\psi'}{\partial q(y)}(q,\lambda,0)
	&= - c(q,\lambda,y)\psi(q,\lambda,y)(1+y)^{-r}\nn
\intertext{and}
\frac{\partial\dot{\psi}}{\partial q(y)}(q,\lambda,0)
	&= \left[\dot{s}(q,\lambda,y)\psi(q,\lambda,y)
		+ s(q,\lambda,y)\dot{\psi}(q,\lambda,y)\right](1+y)^{-r};\nn
\end{align}
this will be used later in Section~\ref{sec:norming-constants}.
\end{remark}


\section{The eigenvalues}
\label{sec:eigenvalues}

We start this section with a initial, relatively crude localization of the eigenvalues.
A slightly stronger result is shown in \cite{lk}, albeit for a more restrictive class
of perturbations.

\begin{lemma}
\label{lem:crude-asymp-eigenvalues}
Suppose $q\in\sA_r$. Given $\epsilon>0$ arbitrarily small, the eigenvalues of
$H_q$ satisfy
\begin{equation*}
\lambda_n(q) = - a_n + O\bigl(n^{-2/3+\epsilon}\bigr)
\end{equation*}
uniformly on bounded subsets of $\sA_r$.
\end{lemma}
\begin{proof}

Let us abbreviate
\begin{equation*}
\psi_0(z) = \psi_0(z,0),\quad  \psi(z) = \psi(q,z,0),\quad \om(z)=\om(q,z).
\end{equation*}
For $\lambda\in\R$, Lemma~\ref{lem:basic-estimates}\ref{eq:psi} implies
\begin{equation*}
\abs{\psi(\lambda)-\psi_0(\lambda)}
	\le e^{C\om(\lambda)}\frac{C\om(\lambda)}{1+\abs{\lambda}^{1/4}}.
\end{equation*}
Given $\epsilon\in(0,1/6)$, set $\delta_n = 4\bigl(\frac32\pi n\bigr)^{-2/3+\epsilon}$.
Let $\cU$ be a bounded subset of $\sA_r$.
In view of Lemma~\ref{lem:about-omega}, there exists $n_1\in\N$ such that
\begin{equation*}
e^{C\om(-a_n \pm \delta_n)} \le 2
\end{equation*}
for all $n\ge n_1$ and $q\in\cU$. By the same token, there exists $n_2\in\N$ such that
\begin{equation*}
2C\om(-a_n \pm \delta_n)\left(2+\abs{-a_n \pm \delta_n}\right)^{(1-3\epsilon)/2} \le 1
\end{equation*}
for all $n\ge n_2$ and $q\in\cU$ (this inequality does not hold if $\epsilon=0$).
Therefore, there exists a sufficiently large $n_3$ such that
\begin{equation*}
\abs{\psi(-a_n\pm\delta_n)-\psi_0(-a_n\pm\delta_n)}
	\le \frac{1}{\abs{-a_n\pm\delta_n}^{3/4-3\epsilon/2}}
	< \frac{2}{\bigl(\frac32\pi n\bigr)^{1/2-\epsilon}}
\end{equation*}
for all $n\ge n_3$.
Now, consider any sequence $\{c_n\}_{n=1}^\infty\subset\R$ such that $\abs{c_n}\le\delta_n$.
Then, recalling \eqref{eq:ai-prime-asymp-negative},
\begin{align*}
\dot{\psi}_0(-a_n + c_n)
	&= - (-a_n + c_n)^{1/4}\left[\sin\left(\tfrac23(-a_n + c_n)^{3/2}-\tfrac14\pi\right)
		+ O(n^{-1})\right]
	\\[1mm]
	&= (-1)^{n+1}\bigl(\tfrac32\pi n\bigr)^{1/6}
		\left[1+O\bigl(n^{-7/12+2\epsilon}\bigr)\right],
\end{align*}
which in turn implies the existence of a $n_4$ such that
\begin{equation*}
\abss{\dot{\psi}_0(-a_n + c_n)} \ge \tfrac12\bigl(\tfrac32\pi n\bigr)^{1/6},
\end{equation*}
for all $n\ge n_4$. As a consequence,
\begin{equation*}
\abs{\psi_0(-a_n\pm\delta_n)}
	= \abss{\dot{\psi}_0(-a_n + c^{\pm}_n)}\abss{\delta_n}
	\ge \frac{2}{\bigl(\frac32\pi n\bigr)^{1/2-\epsilon}},
\end{equation*}
where $c^{+}_n\in (0,\delta_n)$ and $c^{-}_n\in (-\delta_n,0)$.
Therefore,
\begin{equation*}
\abs{\psi(-a_n\pm\delta_n)-\psi_0(-a_n\pm\delta_n)}
	\le \abs{\psi_0(-a_n\pm\delta_n)}
\end{equation*}
for all $n\ge n_5=\max\{n_1,n_2,n_3,n_4\}$ and $q\in\cU$, and this in turn implies
that $\psi(\lambda)$ has a zero in each interval $(-a_n-\delta_n, -a_n+\delta_n)$
for every $n\ge n_5$.

It remains to show that there are no other zeros, indeed, only
one nearby every $-a_n$ for large $n$. For this, let us consider the contours
\begin{equation*}
\cE^m := \left\{z\in\C: \abs{\zeta} = \bigl(m+\tfrac14\bigr)\pi\right\},
\quad
\cE_k := \left\{z\in\C: \abs{\zeta - \bigl(k-\tfrac14\bigr)\pi}=\tfrac{\pi}{2}\right\},
\quad
m,k\in\N.
\end{equation*}
Let $K=\max\{n_5,k_0\}$, where $k_0$ is the integer mentioned
in Lemma~\ref{lem:bound-for-gA}. Then, the said lemma implies
\begin{equation}
\label{eq:pre-rouche}
\abs{\psi(z)-\psi_0(z)} < 8C\omega(z)\abs{\psi_0(z)},
\end{equation}
for all $z\in\cE_k$, $k>K$. Now, increase $K$ so $\cE^{K}$ encloses the (finitely many)
negative zeros of $\psi(z)$ and \eqref{eq:pre-rouche}
also holds in $\cE^{K}$. Increase $K$ one more time (if necessary)
to ensure that $\omega(z)\le (8C)^{-1}$ whenever
$\abs{z}\ge(\frac32 \pi (K+\frac14))^{2/3}$. Then,
\begin{equation*}
\abs{\psi(z)-\psi_0(z)} < \abs{\psi_0(z)},
\quad z\in\cE^K,
\quad z\in\cE_k,\quad k>K.
\end{equation*}
Finally, apply Rouché's theorem.
\end{proof}

Let us define
\[
\eta_k(q,x) = \frac{\psi(q,\lambda_k(q),x)}{\norm{\psi(q,\lambda_k(q),\cdot)}_2}.
\]

\begin{proposition}
\label{lem:eingenvalue-is-real-analytic}
Given $n\in\N$, $\lambda_n:\sA_r\to\R$ is a real analytic map whose gradient is
\begin{equation*}
\frac{\partial\lambda_n}{\partial q(x)} = \eta^2_n(q,x) (1+x)^{-r}.
\end{equation*}
\end{proposition}

\begin{proof}
Fix $q \in \sA_r$. Since
\begin{equation}
\label{eq:the-trick}
\norm{\psi(q,\lambda_n(q),\cdot)}^2_2
	= -\psi'(q,\lambda_n(q),0)\dot{\psi}(q,\lambda_n(q),0) ,
\end{equation}
it follows that $\dot{\psi}(q,\lambda_n(q),0) \ne 0$.
Due to the Implicit Function Theorem (see, for instance, \cite[Appendix~B]{poeschel}),
there exists an open neighborhood $\cU\subset\sA_r$ of $q$ and a unique continuous map
$\mu_n:\cU\to\R$ such that
\begin{equation*}
	\psi(v,\mu_n(v),0)=0,
\quad v \in\cU, \quad  \mu_n(q)=\lambda_n(q).
\end{equation*}
Moreover, $\mu_n$ is real analytic. Since it is unique, necessarily
$\mu_n(q)=\lambda_n(q)$ for all $q \in \cU$. Therefore, $\lambda_n$ is real analytic.

It remains to compute the gradient of $\lambda_n$. Since
\begin{equation*}
\psi(q,\lambda_n(q),0) = 0,
\end{equation*}
we have
\begin{equation*}
\left.\frac{\partial\psi}{\partial q(y)}(q,\lambda,0)\right|_{\lambda=\lambda_n(q)}
	+ \dot{\psi}(q,\lambda_n(q),0)\ \frac{\partial\lambda_n(q)}{\partial q(y)}
	= 0.
\end{equation*}
Recalling \eqref{eq:gradient-psi},
\begin{equation*}
\left.\frac{\partial\psi}{\partial q(y)}(q,\lambda,0)\right|_{\lambda=\lambda_n(q)}
	= s(q,\lambda_n(q),y)\psi(q,\lambda_n(q),y)(1+y)^{-r}.
\end{equation*}
Therefore, the assertion follows after noticing that
\begin{equation*}
s(q,\lambda_n(q),y) = \frac{\psi(q,\lambda_n(q),y)}{\psi'(q,\lambda_n(q),0)}
\end{equation*}
and subsequently using \eqref{eq:the-trick}.
\end{proof}

Let us recall the function $\omega_r:\N\to\R$ defined in \eqref{eq:omega_r}.
As a result of Lemma~\ref{lem:crude-asymp-eigenvalues}, we clearly have
\begin{equation*}
\om(q,\lambda_n(q)) \le C \norm{q}_{\sA_r}\omega_r(n),\quad q\in\sA_r,
\end{equation*}
and
\begin{equation*}
\bom(q,\lambda_n(q)) \le C \norm{q}_{\bsA_r}\omega_r(n),\quad q\in\bsA_r,
\end{equation*}
for some positive constant $C$.

\begin{lemma}
\label{lem:denominator}
\[
\norm{\psi(q,\lambda_n(q),\cdot)}^2_2
	= \bigl(\tfrac32\pi n\bigr)^{1/3} \left[1 + O\bigl(\omega_r(n)\bigr)\right]
\]
uniformly on bounded subsets of $\bsA_r$.
\end{lemma}
\begin{proof}
Let us abbreviate $\lambda_n\defeq\lambda_n(q)$
and consider $\epsilon\in(0,1/6)$ arbitrarily small.

Recalling \eqref{eq:ai-prime-asymp-negative}, we have
\begin{equation*}
\ai'(-\lambda_n)
	= \frac{\lambda_n^{1/4}}{\sqrt{\pi}}
		\left[\sin\left(\tfrac23\lambda_n^{3/2}-\tfrac{\pi}{4}\right)
		+ O(\lambda_n^{-3/2})\right]
\end{equation*}
which, in conjunction with Lemma~\ref{lem:crude-asymp-eigenvalues}, yields
\begin{equation}
\label{eq:one}
\ai'(-\lambda_n)
	= \frac{(-1)^{n+1}}{\sqrt{\pi}}\bigl(\tfrac32\pi n\bigr)^{1/6}
	\left[1 + O\bigl(n^{-2/3+2\epsilon}\bigr)\right]
\end{equation}
uniformly on bounded subsets of $\bsA_r$ (notice that bounded subsets in $\bsA_r$
induces bounded subsets in $\sA_r$).

Next, we recall that
\[
\norm{\psi(q,\lambda_n,\cdot)}^2_2
	= -\psi'(q,\lambda_n,0)\dot{\psi}(q,\lambda_n,0)
\]
whence, using \eqref{eq:psi-0} and
Lemma~\ref{lem:basic-estimates}\ref{eq:psi-prime}-\ref{eq:psi-dot}, we obtain
\begin{equation*}
\norm{\psi(q,\lambda_n,\cdot)}^2_2
	= \pi\left(\ai'(-\lambda_n)\right)^2
		\left[1 + \frac{\varXi'(q,\lambda_n,0)}{\sqrt{\pi}\ai'(-\lambda_n)}
		- \frac{\dot{\varXi}(q,\lambda_n,0)}{\sqrt{\pi}\ai'(-\lambda_n)}
		- \frac{\dot{\varXi}(q,\lambda_n,0)\varXi'(q,\lambda_n,0)}
			{\pi\left(\ai'(-\lambda_n)\right)^2}\right]
\end{equation*}
where division by 0 is avoided if $n$ is assumed sufficiently large. Moreover,
\eqref{eq:bound-psi-prime} implies
\[
\abs{\varXi'(q,\lambda_n,0)}
	\le C\sigma(\lambda_n)\om(q,\lambda_n)
\]
so, in view of \eqref{eq:one},
\begin{equation}
\label{eq:two}
\frac{\varXi'(q,\lambda_n,0)}{\sqrt{\pi}\ai'(-\lambda_n)}
	= O\bigl(\omega_r(n)\bigr)
\end{equation}
uniformly on bounded subsets in $\bsA_r$. Similarly, this time using
\eqref{eq:bound-psi-dot}, one obtains
\begin{equation}
\label{eq:three}
\frac{\dot{\varXi}(q,\lambda_n,0)}{\sqrt{\pi}\ai'(-\lambda_n)}
	= O\bigl(\omega_r(n)\bigr)
\end{equation}
also uniformly on bounded subsets in $\bsA_r$.
The assertion now follows from \eqref{eq:one}, \eqref{eq:two} and \eqref{eq:three}.
\end{proof}

\begin{theorem}
\label{thm:eigenvalues}
\begin{equation}
\label{eq:eigenvalues}
\lambda_n(q)
	= -a_n + \pi \frac{\int_0^\infty \ai^2(x+a_n)q(x)dx}{(-a_n)^{1/2}}
		+ O\bigl(n^{-1/3}\omega_r^2(n)\bigr),
\end{equation}
uniformly on bounded subsets of $\bsA_r$.
\end{theorem}

\begin{proof}
Due to Proposition~\ref{lem:eingenvalue-is-real-analytic},
\begin{equation*}
\frac{d}{dt}\lambda_n(tq)
	= d_{tq}\lambda_n(q)
	= \inner{\frac{\partial\lambda_n}{\partial(tq)}}{q}_{\sA_r}
	= \inner{\eta^2_n(tq,\cdot)}{q}.
\end{equation*}
Here and henceforth we write
\[
\inner{\eta^2_n(tq,\cdot)}{q} \defeq \int_0^\infty \eta^2_n(tq,x)q(x)dx
\]
just for the sake of brevity; we shall see shortly that the integral is indeed finite.
Then,
\begin{equation*}
\lambda_n(q) + a_n
	= \int_0^1 \inner{\eta^2_n(tq,\cdot)}{q} dt.
\end{equation*}

From now on, $\lambda_n = \lambda_n(tq)$. Given a bounded subset in $\bsA_r$, let
$\cU$ be the induced bounded subset in $\sA_r$. Without
loss of generality we can assume $\cU$ is convex with $0\in\cU$ so
$q\in\cU$ implies $tq\in\cU$ for $t\in[0,1]$.
Recalling Lemma~\ref{lem:basic-estimates}\ref{eq:psi}, one can write
\begin{equation*}
\psi^2(tq,\lambda_n,x)
	= \pi \ai^2(x-\lambda_n) + \varPsi_n(x),
\end{equation*}
where
\begin{equation*}
\varPsi_n(x) = 2\sqrt{\pi}\ai(x-\lambda_n)\varXi(tq,\lambda_n,x) + \varXi^2(tq,\lambda_n,x).
\end{equation*}
Because of \eqref{eq:bound-psi}, one can see that
\begin{equation*}
\abs{\varPsi_n(x)}
	\le C \om(tq,\lambda_n)g_A(x-\lambda_n)
	\left(1+\abs{x-\lambda_n}\right)^{-1/2},
\end{equation*}
and this inequality holds uniformly on $\cU\times[0,1]$. Therefore,
\begin{equation*}
\abs{\inner{\varPsi_n}{q}}
	\le \int_0^\infty\abs{\varPsi_n(x)q(x)}dx
	\le C\om^2(tq,\lambda_n)
\end{equation*}
uniformly on $\cU\times[0,1]$. Recalling Lemma~\ref{lem:crude-asymp-eigenvalues},
this in turn implies
\begin{equation*}
\inner{\psi^2(tq,\lambda_n,\cdot)}{q}
	= \pi \inner{\ai^2(\cdot - \lambda_n)}{q} + O\bigl(\omega_r^2(n)\bigr)
\end{equation*}
uniformly on $\cU\times[0,1]$. Moreover,
\begin{equation}
\label{eq:almost-there}
\abs{\inner{\ai^2(\cdot - \lambda_n)}{q}}
	\le C\norm{q}_{\sA_r}\left[\frac{\log(2+\abs{\lambda_n})}{2+\abs{\lambda_n}}\right]^{1/2}
\end{equation}
so Lemma~\ref{lem:denominator} implies
\begin{equation*}
\lambda_n(q) + a_n
	= \pi\frac{\int_0^1\inner{\ai^2(\cdot-\lambda_n(tq))}{q} dt}
		{\bigl(\tfrac32\pi n\bigr)^{1/3}}
		+ O\bigl(n^{-1/3}\omega_r^2(n)\bigr)
\end{equation*}
uniformly on bounded subsets of $\bsA_r$; note that \eqref{eq:almost-there} already implies
\begin{equation}
\label{eq:note}
\lambda_n(q) + a_n = O\bigl(n^{-1/3}\omega_r(n)\bigr).
\end{equation}

The Mean Value Theorem implies
\begin{equation}
\label{eq:mean-value}
\abs{\ai^2(x-\lambda_n) - \ai^2(x+a_n)}
	= 2\abss{\ai(x-s_{n}(x))}\abss{\ai'(x-s_{n}(x))}\abs{\lambda_n + a_n},
\end{equation}
where either
\begin{equation*}
s_{n}(x)\in(-a_n,\lambda_n)\quad \text{or} \quad s_{n}(x)\in(\lambda_n,-a_n);
\end{equation*}
it is not difficult to see that $s_n$ is indeed a continuous function. Therefore,
\begin{multline*}
\int_0^\infty \abs{\ai^2(x-\lambda_n) - \ai^2(x+a_n)}\abs{q(x)} dx
\\
	\le 2\abs{\lambda_n + a_n}\norm{q}_{\sA_r}
		\left(\int_0^\infty \abss{\ai(x-s_{n}(x))}^2
			\abss{\ai'(x-s_{n}(x))}^2(1+x)^{-r} dx \right)^{1/2}
\end{multline*}
but then, recalling \eqref{eq:note} and Lemma~\ref{lem:g_Ag_C}, and noting that $g_A(x-s)\le 1$,
we obtain 
\[
\int_0^\infty \left[\ai^2(x-\lambda_n) - \ai^2(x+a_n)\right]q(x) dx
	= O\bigl(n^{-1/3}\omega_r(n)\bigr).
\]
Therefore,
\begin{equation*}
\lambda_n(q) + a_n
	= \pi\frac{\int_0^1\inner{\ai^2(\cdot+a_n)}{q} dt}
		{\bigl(\tfrac32\pi n\bigr)^{1/3}}
		+ O\bigl(n^{-1/3}\omega_r^2(n)\bigr)
\end{equation*}
uniformly on bounded subsets of $\bsA_r$.
Since
$(-a_n)^{1/2}(\tfrac32\pi n)^{-1/3} = 1 + O(n^{-1})$, the stated result follows.
\end{proof}


\section{More auxiliary results}
\label{sec:auxiliary}

In what follows we use the abbreviated notation
\[
\lambda_n = \lambda_n(q);\quad
\psi_n = \psi(q,\lambda_n,0),\quad
\psi_n' = \psi'(q,\lambda_n,0),\quad
\dot{\psi}_n = \dot{\psi}(q,\lambda_n,0),\quad\text{et cetera.}
\]
We also introduce the notation
\[
\alpha_n = \psi_0(\lambda_n,0),\quad
\alpha_n' = \psi_0'(\lambda_n,0),\quad
\beta_n = \theta_0(\lambda_n,0),\quad\text{and}\quad
\beta_n' = \theta_0'(\lambda_n,0).
\]
The following asymptotics follow easily from combining \eqref{eq:eigenvalues} with
\eqref{eq:ai-asymp-negative}, \eqref{eq:ai-prime-asymp-negative}, \eqref{eq:bi-1}
and \eqref{eq:bi-prime-asymp-negative}.

\begin{lemma}
\label{lem:alpha-etcetera}
The following asymptotic expressions:
\begin{align*}
\alpha_n  &= O\bigl(n^{-1/6}\omega_r(n)\bigr),
\\[1mm]
\alpha_n' &= (-1)^{n+1}\bigl(\tfrac32\pi n\bigr)^{1/6}
			\left[1 + O\bigl(\omega_r^2(n)\bigr)\right],
\\[1mm]
\beta_n   &= (-1)^{n}\bigl(\tfrac32\pi n\bigr)^{-1/6}
			\left[1 + O\bigl(\omega_r^2(n)\bigr)\right],
\\[1mm]
\beta_n'  &= O\bigl(n^{1/6}\omega_r(n)\bigr),
\end{align*}
holds uniformly on bounded subsets of $\bsA_r$.
\end{lemma}

\begin{lemma}
\label{lem:first-parenthesis}
\[
\frac{\varXi_n'}{\psi_n'} - \frac{\dot{\varXi}_n}{\dot{\psi}_n}
	= O\bigl(\omega_r^2(n)\bigr),
\]
uniformly on bounded subsets of $\bsA_r$.
\end{lemma}
\begin{proof}
First, let us note that
\begin{equation}
\label{eq:quotients}
\frac{\varXi_n'}{\alpha_n'} = O\bigl(\omega_r(n)\bigr)\quad\text{and}\quad
\frac{\dot{\varXi}_n}{\alpha_n'} = O\bigl(\omega_r(n)\bigr);
\end{equation}
this follows from Lemma~\ref{lem:basic-estimates}\ref{eq:psi-prime}-\ref{eq:psi-dot} and
Lemma~\ref{lem:alpha-etcetera}. Then,
\begin{align*}
\frac{\varXi_n'}{\psi_n'} - \frac{\dot{\varXi}_n}{\dot{\psi}_n}
	&= \frac{\varXi_n'}{\alpha_n'\left(1+\frac{\varXi_n'}{\alpha_n'}\right)}
		+ \frac{\dot{\varXi}_n}{\alpha_n'\left(1-\frac{\dot{\varXi}_n}{\alpha_n'}\right)}
\\[1mm]
	&= \frac{\dot{\psi}_{1,n}^\text{res} + {\varXi^{(2)\prime}_n}
		+ \dot{\varXi}^{(2)}_n}{\alpha_n'} + O\bigl(\omega_r^2(n)\bigr),
\end{align*}
where we have made use of Remark~\ref{rem:more-about-psi}. Now,
\[
\dot{\psi}_{1,n}^\text{res}
	= - \beta_n \int_0^\infty \psi_0^2(\lambda_n,x)q'(x)dx
		+ \alpha_n \int_0^\infty \psi_0(\lambda_n,x)\theta_0(\lambda_n,x)q'(x)dx
\]
so, by combining Lemma~\ref{lem:about-omega} and
Lemma~\ref{lem:alpha-etcetera}, we obtain
\[
\frac{\dot{\psi}_{1,n}^\text{res}}{\alpha_n'} = O\bigl(n^{-1/3}\omega_r(n)\bigr)
\]
uniformly on bounded sets of $\bsA_r$.
Also, from Remark~\ref{rem:more-about-psi}, we conclude that
\[
\frac{\varXi^{(2)\prime}_n}{\alpha_n'}
		+ \frac{\dot{\varXi}^{(2)}_n}{\alpha_n'} = O\bigl(\omega_r^2(n)\bigr).
\]
Now the stated result follows immediately.
\end{proof}

\begin{lemma}
\label{lem:second-parenthesis}
\[
\frac{\dot{\psi}_n'}{\psi_n'} - \frac{\ddot{\psi}_n}{\dot{\psi}_n}
	= O\bigl(n^{1/3}\omega_r^2(n)\bigr),
\]
uniformly on bounded subsets of $\bsA_r$.
\end{lemma}
\begin{proof}
We have
\begin{multline*}
\frac{\dot{\psi}_n'}{\psi_n'} - \frac{\ddot{\psi}_n}{\dot{\psi}_n}
	= \left[-\frac{\lambda_n\alpha_n}{(\alpha_n')^2}\left(\varXi_n'+\dot{\varXi}_n\right)\right.
\\
	\left.+\frac{1}{\alpha_n'}\left(\dot{\varXi}_n'+\ddot{\varXi}_n\right)
	+\frac{1}{(\alpha_n')^2}\left(\varXi_n'\ddot{\varXi}_n
		-\dot{\varXi}_n\dot{\varXi}_n'\right)\right]
	\left[1 + O\bigl(\omega_r(n)\bigr)\right].
\end{multline*}
Clearly,
\[
\frac{\lambda_n\alpha_n}{(\alpha_n')^2}\left(\varXi_n'+\dot{\varXi}_n\right)
	= \lambda_n\frac{\alpha_n}{\alpha_n'}
		\frac{\dot{\psi}_{1,n}^\text{res}
			+ {\varXi^{(2)\prime}_n} + \dot{\varXi}^{(2)}_n}{\alpha_n'}
	= O\bigl(\omega_r^2(n)\bigr).
\]
Since
\[
\dot{\varXi}_n'+\ddot{\varXi}_n = \dot{\psi}_n' + \ddot{\psi}_n,
\]
Lemma~\ref{lem:even-more-about-psi} implies
\[
\frac{1}{\alpha_n'}\left(\dot{\varXi}_n'+\ddot{\varXi}_n\right)
	= O\bigl(\omega_r(n)\bigr).
\]
Finally, Lemma~\ref{lem:even-more-about-psi} and \eqref{eq:double-dot-Xi} imply
\[
\dot{\varXi}'_n = O\bigl(n^{1/2}\omega_r(n)\bigr)\quad\text{and}\quad
\ddot{\varXi}_n = O\bigl(n^{1/2}\omega_r(n)\bigr),
\]
hence
\[
\frac{1}{(\alpha_n')^2}\left(\varXi_n'\ddot{\varXi}_n
		-\dot{\varXi}_n\dot{\varXi}_n'\right)= O\bigl(n^{1/3}\omega_r^2(n)\bigr).\qedhere
\]
\end{proof}

\begin{lemma}
\label{lem:alphas-and-betas}
Suppose $q\in\bsA_r$. Define $A^0_n(q,x)$ by the rule
\begin{equation}
\label{eq:A-super-0}
(1+x)^r A^0_n(q,x)
	= 2 \frac{\beta_n}{\alpha_n'}
		\psi_0(\lambda_n,x)\psi_0'(\lambda_n,x)
		- \frac{\alpha_n}{\alpha_n'}
		\left[\theta_0'(\lambda_n,x)\psi_0(\lambda_n,x)
			+ \theta_0(\lambda_n,x)\psi_0'(\lambda_n,x)\right].
\end{equation}
Then $A_n^0(q,\cdot)\in\sA_r$ and
\begin{equation*}
\int_0^1\inner{A_n^0(tq,\cdot)}{q}_{\sA_r}dt
	= -2 \pi \frac{\int_0^\infty \ai(x+a_n)\ai'(x+a_n)q(x)dx}{\bigl(\tfrac32\pi n\bigr)^{1/3}}
			+ O\bigl(n^{-2/3}\omega_r(n)\bigr)
\end{equation*}
uniformly on bounded subsets of $\bsA_r$.
\end{lemma}
\begin{proof}
Let us abbreviate $\lambda_n \defeq \lambda_n(tq)$. An integration by parts yields\footnote{Recall
that $q\in\bsA_r$ implies $q$ bounded.}
\begin{equation}
\label{eq:integration-by-parts}
\inner{A_n^0(tq,\cdot)}{q}_{\sA_r}
=	- \frac{\beta_n}{\alpha_n'}\int_0^\infty \psi_0^2(\lambda_n,x) q'(x) dx
	+ \frac{\alpha_n}{\alpha_n'}\int_0^\infty \theta_0(\lambda_n,x)\psi_0(\lambda_n,x) q'(x) dx,
\end{equation}
where, because of Lemma~\ref{lem:alpha-etcetera},
\begin{equation*}
\frac{\beta_n}{\alpha_n'}
	= - \bigl(\tfrac32\pi n\bigr)^{-1/3}\left[1 + O\bigl(\omega_r^2(n)\bigr)\right],
\quad
\frac{\alpha_n}{\alpha_n'}
	= O\bigl(n^{-1/3}\omega_r(n)\bigr),
\end{equation*}
uniformly on bounded subsets of ${\bsA}_r$.
Resorting to Lemma~\ref{lem:g_Ag_C} and Lemma~\ref{lem:about-omega}, we obtain
\[
\abs{\int_0^\infty \psi_0^2(\lambda_n,x) q'(x) dx}
\le C \bom(q,\lambda_n),
\]
and a similar bound holds for the second integral in \eqref{eq:integration-by-parts}.
Therefore,
\[
\frac{\beta_n}{\alpha_n'}\int_0^\infty \psi_0^2(\lambda_n,x) q'(x) dx
	= O\bigl(n^{-1/3}\omega_r(n)\bigr)
\]
and
\[
\frac{\alpha_n}{\alpha_n'}\int_0^\infty \theta_0(\lambda_n,x)\psi_0(\lambda_n,x) q'(x) dx
	= O\bigl(n^{-1/3}\omega_r^2(n)\bigr).
\]
On the other hand,
\begin{multline*}
\int_0^\infty \psi_0^2(\lambda_n,x) q'(x) dx
\\
	= - 2\int_0^\infty \psi_0(-a_n,x)\psi_0'(-a_n,x) q(x) dx
		+ \int_0^\infty \left[\psi_0^2(\lambda_n,x) - \psi_0^2(-a_n,x)\right] q'(x) dx.
\end{multline*}
An argument like in the proof of Theorem~\ref{thm:eigenvalues} (see the exposition after
\eqref{eq:mean-value}) implies
\[
\frac{\beta_n}{\alpha_n'}\int_0^\infty
	\left[\psi_0^2(\lambda_n,x) - \psi_0^2(-a_n,x)\right] q'(x) dx
		= O\bigl(n^{-2/3}\omega_r(n)\bigr),
\]
hence concluding the proof.
\end{proof}

\begin{lemma}
\label{lem:Lambda-one-and-two}
Define
\begin{equation}
\label{eq:Lambda_1}
\varLambda_1(q,z,x)
	\defeq c_0(z,x)\varXi(q,z,x) + \psi_0(z,x)\varUpsilon_c(q,z,x)
		+ \varXi(q,z,x)\varUpsilon_c(q,z,x),
\end{equation}
and
\begin{multline*}
\varLambda_2(q,z,x)
	\defeq \dot{s}_0(z,x)\varXi(q,z,x) + s_0(z,x)\dot{\varXi}(q,z,x)
		+ \psi_0(z,x)\dot{\varUpsilon}_s(q,z,x)
\\[1mm]
		+ \dot{\psi}_0(z,x)\varUpsilon_s(q,z,x)
		+ \varUpsilon_s(q,z,x)\dot{\varXi}(q,z,x) + \dot{\varUpsilon}_s(z,x)\varXi(q,z,x).
\end{multline*}
Then,
\[
\inner{\varLambda_1(q,\lambda_n(q),\cdot)}{v} = O\bigl(n^{1/6}\omega_r^2(n)\bigr),
\]
uniformly on bounded subsets of $\bsA_r\times\sA_r$. Also,
\[
\inner{\varLambda_2(q,\lambda_n(q),\cdot)}{v} = O\bigl(n^{1/6}\omega_r^2(n)\bigr),
\]
uniformly on bounded subsets of $\bsA_r\times\sA_r$.
\end{lemma}
\begin{proof}
As before, we write $\lambda_n \defeq \lambda_n(q)$.
$q\in\bsA_r$ and $v\in\sA_r$,
and consider the first term in \eqref{eq:Lambda_1}. We have
\[
\inner{c_0(\lambda_n,\cdot)\varXi(q,\lambda_n,\cdot)}{v}
	= \beta_n'\inner{\psi_0(\lambda_n,\cdot)\varXi(q,\lambda_n,\cdot)}{v}
		- \alpha_n'\inner{\theta_0(\lambda_n,\cdot)\varXi(q,\lambda_n,\cdot)}{v}.
\]
Clearly,
\begin{align*}
\abs{\inner{\psi_0(\lambda_n,\cdot)\varXi(q,\lambda_n,\cdot)}{v}}
	&\le \int_0^\infty\abs{\psi_0(\lambda_n,x)}\abs{\varXi(q,\lambda_n,x)}\abs{v(x)}dx
\\
	&\le C\om(q,\lambda_n) \int_0^\infty
		\frac{\abs{v(x)}}{\sigma(x-\lambda_n)^2}dx
	 = C \om(q,\lambda_n)\om(v,\lambda_n)
\end{align*}
and, similarly,
\[
\abs{\inner{\theta_0(\lambda_n,\cdot)\varXi(q,\lambda_n,\cdot)}{v}}
	\le C \om(q,\lambda_n)\om(v,\lambda_n).
\]
Therefore, in view of Lemma~\ref{lem:alpha-etcetera},
\begin{equation*}
\inner{c_0(\lambda_n,\cdot)\varXi(q,\lambda_n,\cdot)}{v}
	= O\bigl(n^{1/6}\omega_r^2(n)\bigr)
\end{equation*}
uniformly on bounded subsets of $\bsA_r\times\sA_r$. Next, we note that
\begin{equation}
\label{eq:saving-ineq}
g_A(x-\lambda_n)\ch(\lambda_n,x)
	= g_B(-\lambda_n) g_A^2(x-\lambda_n) + g_A(-\lambda_n).
\end{equation}
Then, resorting to Lemma~\ref{lem:stupid-lemma} and noticing that $g_A(\lambda)\le 1$
on the whole real line, we obtain
\begin{equation*}
\abs{\inner{\psi_0(\lambda_n,\cdot)\varUpsilon_c(q,\lambda_n,\cdot)}{v}}
	\le C \sigma(\lambda_n)\om(q,\lambda_n)\om(v,\lambda_n).
\end{equation*}
Thus,
\[
\inner{\psi_0(\lambda_n,\cdot)\varUpsilon_c(q,\lambda_n,\cdot)}{v}
	= O\bigl(n^{1/6}\omega_r^2(n)\bigr)
\]
uniformly on bounded subsets of $\bsA_r\times\sA_r$.
Finally, using similar arguments,
\begin{equation*}
\abs{\inner{\varXi(q,\lambda_n,\cdot)\varUpsilon_c(q,\lambda_n,\cdot)}{v}}
	\le C \sigma(\lambda_n)\om^2(q,\lambda_n)\om(v,\lambda_n)
\end{equation*}
so
\[
\inner{\varXi(q,\lambda_n,\cdot)\varUpsilon_c(q,\lambda_n,\cdot)}{v}
	= O\bigl(n^{1/6}\omega_r^3(n)\bigr).
\]
The proof of the statement concerning $\varLambda_1$ is now complete.

Let us look at the first term in $\varLambda_2$.
Recalling \eqref{eq:saving-ineq} and Lemma~\ref{lem:stupid-lemma}, we obtain
\begin{equation*}
\abs{\inner{\dot{s}_0(z,\cdot)\varXi(q,z,\cdot)}{v}}
	\le C \norm{q}_{\bsA_r}\norm{v}_{\sA_r}\left(\frac{\om_r(n)}{\sigma(\lambda_n)}
	+ \sigma(\lambda_n)\om_r^2(n)\right),
\end{equation*}
where we also use that $\norm{v}_1\le (r-1)^{-1/2}\norm{v}_{\sA_r}$. But since
$n^{-1/6}\om_r(n)\le n^{1/6}\om_r^2(n)$, it follows that
\[
\inner{\dot{s}_0(z,\cdot)\varXi(q,z,\cdot)}{v} = O\bigl(n^{1/6}\omega_r^2(n)\bigr),
\]
uniformly on bounded subsets of $\bsA_r\times\sA_r$.

The third and last terms in $\varLambda_2$ are dealt with in essentially the same fashion
as above. The second, fourth and fifth terms are similar to those in $\varLambda_1$.
\end{proof}

\begin{lemma}
\label{lem:last-piece}
Assume $q\in\bsA_r$. Define
\[
\varLambda_n(q,x)
	= \left[\frac{\varLambda_1(q,\lambda_n(q),x)}{\psi_n'}
		+ \frac{\varLambda_2(q,\lambda_n(q),x)}{\dot{\psi}_n}\right](1+x)^{-r}.
\]
Then, $\varLambda_n(q,\cdot)\in\sA_r$ and
\[
\inner{\varLambda_n(q,\cdot)}{v}_{\sA_r} = O\bigl(n^{-1/3}\omega_r^2(n)\bigr)
\]
uniformly for $q$ and $v$ in bounded subsets of $\bsA_r$.
\end{lemma}
\begin{proof}
Let us abbreviate $\varLambda_1(x) = \varLambda_1(q,\lambda_n,x)$,
$\varLambda_2(x) = \varLambda_2(q,\lambda_n,x)$ and so on.
A computation\footnote{It is perhaps worth mentioning the identities
\[
\frac{1}{\psi_n'} = \frac{1}{\alpha_n'}\left(1 - \frac{\varXi_n'}{\psi_n'}\right)
\quad\text{and}\quad
\frac{1}{\dot{\psi}_n} = -\frac{1}{\alpha_n'}\left(1 - \frac{\dot{\varXi}_n}{\dot{\psi}_n}\right).
\]
}
yields
\begin{equation}
\label{eq:super-mess}
\frac{\varLambda_1(x)}{\psi_n'}
		+ \frac{\varLambda_2(x)}{\dot{\psi}_n}
		= \frac{\varLambda_1(x) - \varLambda_2(x)}{\alpha_n'}
		+ \left[\frac{\dot{\varXi}_n}{\alpha_n'\dot{\psi}_n}\varLambda_2(x)
				- \frac{\varXi'_n}{\alpha_n'\psi_n'}\varLambda_1(x)\right].
\end{equation}
Suppose $v\in\bsA_r$. Then,
\[
\inner{\frac{\dot{\varXi}_n}{\alpha_n'\dot{\psi}_n}\varLambda_2
		- \frac{\varXi'_n}{\alpha_n'\psi_n'}\varLambda_1}{v}
	= \frac{\dot{\varXi}_n}{\alpha_n'\dot{\psi}_n}\inner{\varLambda_2}{v}
		- \frac{\varXi'_n}{\alpha_n'\psi_n'}\inner{\varLambda_1}{v}
	= O\bigl(n^{-1/6}\omega_r^3(n)\bigr)
\]
uniformly on bounded subsets of $\bsA_r\times\bsA_r$, as it follows from \eqref{eq:quotients}
and Lemma~\ref{lem:Lambda-one-and-two}.

Next, let us look at the first term in \eqref{eq:super-mess}. Here, another computation yields
\begin{multline*}
\varLambda_1(x) - \varLambda_2(x)
	= s_0'(x)\varXi(x) - s_0(x)\dot{\varXi}(x)
		+ \psi_0'(x)\varUpsilon_s(x) - \psi_0(x)\dot{\varUpsilon}_s(x)
\\[1mm]
		+ \psi_0(x)\varUpsilon_c(x) + \varUpsilon_c(x)\varXi(x)
		- \dot{\varUpsilon}_s(x)\varXi(x) - \varUpsilon_s(x)\dot{\varXi}(x).
\end{multline*}
Now, resorting to Remark~\ref{rem:more-about-psi},
\begin{equation*}
s_0'(x)\varXi(x) - s_0(x)\dot{\varXi}(x)
	= (s_0\psi_1)'(x) - s_0(x)\dot{\psi}_1^\text{res}(x)
		+ s_0'(x)\varXi^{(2)}(x) - s_0(x)\dot{\varXi}^{(2)}(x)
\end{equation*}
so, since $v'\in\sA_r$,
\begin{equation*}
\inner{s_0'\varXi - s_0\dot{\varXi}}{v}
	= - \inner{s_0\varXi}{v'} - \inner{s_0\dot{\psi}_1^\text{res}}{v}
		+ \inner{s_0'\varXi^{(2)}}{v} - \inner{s_0\dot{\varXi}^{(2)}}{v}.
\end{equation*}
Hence,
\begin{equation*}
\frac{1}{\alpha_n'}\inner{s_0'\varXi - s_0\dot{\varXi}}{v}
	= O\bigl(n^{-1/3}\omega_r^2(n)\bigr)
\end{equation*}
uniformly on bounded subsets of $\bsA_r\times\bsA_r$. Analogously, resorting to
Remark~\ref{rem:even-more-about-s} and integration by parts, we obtain
\begin{equation*}
\inner{\psi_0'\varUpsilon_s - \psi_0\dot{\varUpsilon}_s + \psi_0\varUpsilon_c}{v}
	= - \inner{\psi_0s_1}{v'} - \inner{\psi_0\dot{s}_1^\text{res}}{v}
		+ \inner{\psi_0'\varUpsilon^{(2)}_s}{v} - \inner{\psi_0\dot{\varUpsilon}^{(2)}_s}{v}
		+ \inner{\psi_0\varUpsilon^{(2)}_c}{v},
\end{equation*}
thus yielding
\begin{equation*}
\frac{1}{\alpha_n'}
	\inner{\psi_0'\varUpsilon_s - \psi_0\dot{\varUpsilon}_s + \psi_0\varUpsilon_c}{v}
	= O\bigl(n^{-1/3}\omega_r^2(n)\bigr).
\end{equation*}
Finally, since
\begin{equation*}
\varUpsilon_c(x)\varXi(x) - \dot{\varUpsilon}_s(x)\varXi(x)
	= s_1(x)\varXi(x) - \dot{s}_1^\text{res}(x)\varXi(x)
		+ \varUpsilon^{(2)}_c\varXi(x) - \dot{\varUpsilon}^{(2)}_s\varXi(x),
\end{equation*}
the same line of reasoning already used implies
\begin{equation*}
\frac{1}{\alpha_n'}\inner{\varUpsilon_c\varXi - \dot{\varUpsilon}_s\varXi
	- \varUpsilon_s\dot{\varXi}}{v} = O\bigl(n^{-1/3}\omega_r^2(n)\bigr),
\end{equation*}
uniformly on bounded subsets of $\bsA_r\times\bsA_r$. The proof is now complete.
\end{proof}


\section{The norming constants}
\label{sec:norming-constants}

\begin{proposition}
\label{lem:norming-constant-real-analytic}
Given $n\in\N$, $\kappa_n:\sA_r\to\R$ is a real analytic map whose gradient is
\begin{equation*}
\frac{\partial\kappa_n}{\partial q(x)}
	= \left(\frac{1}{\psi_n'}\frac{\partial\psi'_n}{\partial q(x)}
		- \frac{1}{\dot{\psi}_n}\frac{\partial\dot{\psi}_n}{\partial q(x)}\right)
		+ \left(\frac{\dot{\psi}_n'}{\psi_n'} - \frac{\ddot{\psi}_n}{\dot{\psi}_n}\right)
		\frac{\partial\lambda_n}{\partial q(x)}.
\end{equation*}
\end{proposition}
\begin{proof}
In view of Remark~\ref{rem:everything-has-gradient}, by definition $\kappa_n$ is
a composition of continuously differentiable maps.
\end{proof}

\begin{theorem}
\label{thm:norming-constants}
Suppose $q\in\bsA_r$. Then,
\begin{equation*}
\kappa_n(q)
	= - 2\pi \frac{\int_0^\infty \ai(x+a_n)\ai'(x+a_n)q(x)dx}{(-a_n)^{1/2}}
		+ O\bigl(\omega_r^3(n)\bigr).
\end{equation*}
Moreover, this formula holds uniformly on bounded subsets of $\bsA_r$.
\end{theorem}

\begin{proof}
Let us abbreviate
\begin{equation*}
\frac{\partial\kappa_n}{\partial q(x)}
	= A_n(q,x) + B_n(q)\frac{\partial\lambda_n}{\partial q(x)},
\end{equation*}
where the definitions of $A_n(q,x)$ and $B_n(q)$ are self-evident.
Since $\kappa_n(0)=0$, one has
\begin{align}
\kappa_n(q)
	&= \int_0^1 \inner{\frac{\partial\kappa_n}{\partial(tq)}}{q}_{\sA_r} dt\nonumber
	\\
	&= \int_0^1 \inner{A_n(tq,\cdot)}{q}_{\sA_r} dt
		+ \int_0^1 B_n(tq)\inner{\eta^2_n(tq,\cdot)}{q} dt.\label{eq:integral-kappa}
\end{align}

Let us look at the first term above. Recalling \eqref{eq:A-super-0},
a routine computation shows that
\[
A_n(q,x) = A^0_n(q,x) + \varDelta_n(q,x),
\]
where
\begin{multline}
\varDelta_n(q,x)
	= \frac{(c_0\psi_0)(\lambda_n,x)}{\alpha_n'}
		\left(\frac{\varXi_n'}{\psi_n'} - \frac{\dot{\varXi}_n}{\dot{\psi}_n}\right)(1+x)^{-r}
\\
	+ \frac{(s_0\psi_0'+ s_0'\psi_0)(\lambda_n,x)}{\alpha_n'}
		\frac{\dot{\varXi}_n}{\dot{\psi}_n}(1+x)^{-r}
	- \varLambda_n(q,x);\label{eq:mess}
\end{multline}
the last term has been defined in Lemma~\ref{lem:last-piece}.

In order to simplify the ongoing discussion, let us write
\[
\varDelta_n(q,x) = 1_n(x) + 2_n(x) + 3_n(x),
\]
where each term is defined by the corresponding one (and arranged in the same order as)
in \eqref{eq:mess}.
Now,
\[
\inner{1_n}{q}_{\sA_r} = O\bigl(\omega_r^3(n)\bigr)
\]
uniformly on bounded sets of $\bsA_r$ because
\[
(c_0\psi_0)(\lambda_n,x)
	= \beta_n'\psi_0^2(\lambda_n,x) - \alpha_n'(\psi_0\theta_0)(\lambda_n,x),
\]
and
\[
\inner{\psi_0^2(\lambda_n,\cdot)}{q} = O\bigl(\omega_r(n)\bigr),\quad
\inner{(\psi_0\theta_0)(\lambda_n,\cdot)}{q} = O\bigl(\omega_r(n)\bigr),
\]
according to an argument already used in the proof of Lemma~\ref{lem:first-parenthesis}.
Analogously,
\[
(s_0\psi_0)(\lambda_n,x)
	= - \beta_n \psi_0^2(\lambda_n,x) + \alpha_n (\psi_0\theta_0)(\lambda_n,x)
\]
so
\[
\inner{(s_0\psi_0' + s_0'\psi_0)(\lambda_n,\cdot)}{q}
	= - \beta_n \inner{\psi_0^ 2(\lambda_n,\cdot)}{q'}
		+ \alpha_n \inner{(\psi_0\theta_0)(\lambda_n,\cdot)}{q'}.
\]
This implies
\[
\inner{2_n}{q}_{\sA_r} = O\bigl(n^{-1/3}\omega_r^2(n)\bigr),
\]
also uniformly on bounded sets of $\bsA_r$, due to the preceding argument along with
\eqref{eq:quotients}.
Finally, the last term has been dealt with in Lemma~\ref{lem:last-piece} so
\[
\inner{3_n}{q}_{\sA_r}
	= - \inner{\varLambda_n(q,\cdot)}{q}_{\sA_r}
	= O\bigl(n^{-1/3}\omega_r^2(n)\bigr).
\]

It remains to look at the second term in \eqref{eq:integral-kappa}. From the proof
of Theorem~\ref{thm:eigenvalues}, it follows that
\[
\inner{\eta^2_n(tq,\cdot)}{q} = O\bigl(n^{-1/3}\omega_r(n)\bigr)
\]
uniformly in bounded subsets of $\bsA_r$. 
Moreover, Lemma~\ref{lem:second-parenthesis} implies
\[
B_n(tq) = O\bigl(n^{1/3}\omega_r^2(n)\bigr)
\]
also uniformly as above. Therefore,
\[
\int_0^1 B_n(tq)\inner{\eta^2_n(tq,\cdot)}{q} dt = O\bigl(\omega_r^3(n)\bigr).
\]

Summing up, by recalling Lemma~\ref{lem:alphas-and-betas} we obtain
\begin{equation*}
\kappa_n(q)
	= - 2\pi \frac{\int_0^\infty \ai(x+a_n)\ai'(x+a_n)q(x)dx}{\bigl(\tfrac32\pi n\bigr)^{1/3}}
		+ O\bigl(\omega_r^3(n)\bigr),
\end{equation*}
but this yields the desired result because $(-a_n)^{1/2}(\tfrac32\pi n)^{-1/3} = 1 + O(n^{-1})$.
\end{proof}


\section*{Acknowledgments}

This research is based upon work supported by the National Scientific and Technical
Research Council (CONICET, Argentina) under grant PIP 11220200102127CO, and
Universidad Nacional del Sur (Argentina) under grant PGI 24/L117. The authors
thank Pablo Panzone for helpful comments.

\section*{Data availability}

Data sharing not applicable to this article as no datasets were generated or
analyzed during the current study.


\appendix

\section{Some results concerning the Airy functions}

The Airy functions of the first and second kind, $\ai$ and $\bi$, are real entire
solutions to the differential equation
\begin{equation}
\label{eq:airy-ode}
\frac{d^2\varphi}{dz^2} = z\varphi\quad (z\in\C).
\end{equation}
They are linearly independent; indeed $W_z\{\ai,\bi\} = \ai(z)\bi'(z) - \ai'(z)\bi(z) = 1/\pi$
\cite[\S 9.2]{nist}. Another pair of linearly independent solutions to \eqref{eq:airy-ode}
is given by the linear combinations
\begin{equation}
\label{eq:ci}
\ci_\pm(z)
	= \bi(z) \mp i \ai(z)
	= 2e^{\mp i\pi/6}\ai(ze^{\mp i2\pi/3});
\end{equation}
they obey $W_z\{\ai,\ci_\pm\} = 1/\pi$.

According to \cite[\S 9.7]{nist}, the function $\ai$ satisfies the asymptotic expansions
\begin{gather}
\ai(z) = \frac{e^{-\zeta}}{2\sqrt{\pi}z^{1/4}}\left[1+O(\zeta^{-1})\right],
\quad \abs{\arg(z)}\le\pi-\delta,\label{eq:ai-asymp-positive}
\intertext{and}
\ai(-z)
	= \frac{1}{\sqrt{\pi}z^{1/4}}\left[\cos\left(\zeta - \tfrac{\pi}{4}\right)
		+ O\left(\zeta^{-1}e^{\abs{\im\zeta}}\right)\right],
\quad \abs{\arg(z)}\le\tfrac{2\pi}{3} - \delta,\label{eq:ai-asymp-negative}
\end{gather}
as $\abs{z}\to\infty$; here $\zeta =\frac23 z^{3/2}$ whose branch cut is assumed to
be along $\R_-$. Its derivative obeys
\begin{gather}
\ai'(z) = -\frac{z^{1/4}e^{-\zeta}}{2\sqrt{\pi}}\left[1+O(\zeta^{-1})\right],
\quad \abs{\arg(z)}\le\pi-\delta,\label{eq:ai-prime-asymp-positive}
\intertext{and}
\ai'(-z)
	= \frac{z^{1/4}}{\sqrt{\pi}}\left[\sin\left(\zeta - \tfrac{\pi}{4}\right)
		+ O\left(\zeta^{-1}e^{\abs{\im\zeta}}\right)\right],
\quad \abs{\arg(z)}\le\tfrac{2\pi}{3} - \delta.\label{eq:ai-prime-asymp-negative}
\end{gather}
As for the function $\bi$, one has
\begin{gather}
\bi(z)
	= \frac{e^{\zeta}}{\sqrt{\pi}z^{1/4}}
		\left(1+O(\zeta^{-1})\right),
		\quad \abs{\arg(z)}<\tfrac{\pi}{3}-\delta,\nn
\\[1mm]
\bi(-z)
	= \frac{1}{\sqrt{\pi}z^{1/4}}
	  \left[-\sin\left(\zeta-\tfrac{\pi}{4}\right)
	  	+O\left(\zeta^{-1}e^{\abs{\im\zeta}}\right)\right],
	  \quad \abs{\arg(z)}\le\tfrac{2\pi}{3} - \delta,\label{eq:bi-1}
\intertext{and}
\bi(ze^{\pm i\pi/3})
	= \sqrt{\frac{2}{\pi}}\frac{e^{\pm i\pi/6}}{z^{1/4}}
	  \left[\cos(\zeta-\tfrac{\pi}{4}\mp i\tfrac{\log 2}{2})
	  	+ O\left(\zeta^{-1}e^{\abs{\im\zeta}}\right)
	  \right],
\quad \abs{\arg(z)}\le\tfrac{2\pi}{3}-\delta.\nn
\end{gather}
In addition,
\begin{equation}
\label{eq:bi-prime-asymp-negative}
\bi'(-z)
	= \frac{z^{1/4}}{\sqrt{\pi}}
		\left[\cos\left(\zeta-\tfrac{\pi}{4}\right)
		+O\left(\zeta^{-1}e^{\abs{\im\zeta}}\right)\right],
	\quad \abs{\arg(z)}\le\tfrac{2\pi}{3} - \delta.
\end{equation}
All these expansions are uniform for any given small $\delta>0$ and (say) $\abs{z}\ge 1$.

\begin{lemma}
\label{lem:g_Ag_C}
Define $g_A(z) = \exp(-\tfrac23\re z^{3/2})$ and $g_B(z) = 1/g_A(z)$.
There exists a constant $C_0>0$ such that
\begin{equation*}
\abss{\ai(z)}
\le C_0\frac{g_A(z)}{1+\abs{z}^{1/4}} \quad\text{and}\quad
\abss{\ai'(z)}
\le C_0 \bigl(1 + \abs{z}^{1/4}\bigr) g_A(z),
\end{equation*}
hence
\begin{equation*}
\abss{\ci_\pm(z)}
\le 2C_0\frac{g_B(z)}{1+\abs{z}^{1/4}} \quad\text{and}\quad
\abss{\ci_\pm'(z)}
\le 2C_0 \bigl(1 + \abs{z}^{1/4}\bigr) g_B(z),
\end{equation*}
for all $z\in\C$. As a consequence,
\begin{equation*}
\abss{\bi(z)}
\le 2C_0\frac{g_B(z)}{1+\abs{z}^{1/4}} \quad\text{and}\quad
\abss{\bi'(z)}
\le 2C_0 \bigl(1 + \abs{z}^{1/4}\bigr) g_B(z),
\end{equation*}
also for all $z\in\C$.
\end{lemma}
\begin{proof}
Set $\delta=\pi/3$ in \eqref{eq:ai-asymp-positive} and \eqref{eq:ai-asymp-negative}.
Since $\text{\rm Ai}(z)$ is an entire function, it follows that there exists $C_0>0$
such that
\[
\abs{\ai(z)}
	\le \frac{C_0}{1+\abs{z}^{1/4}}\times
	\begin{cases}
		\exp(-\tfrac23 \re z^{3/2}),&
		\arg(z)\in [-\frac{2\pi}{3},\frac{2\pi}{3}],
		\\[1mm]
		\exp(\tfrac23\abss{\im(-z)^{3/2}}),&
		\arg(z)\in(-\pi,-\frac{2\pi}{3})\cup(\frac{2\pi}{3},\pi].
		\end{cases}
\]
Thus, the bound on $\ai(z)$ follows after noticing that
$\abss{\im(-z)^{3/2}}=\abss{\re z^{3/2}}$ and $\abss{\re z^{3/2}}=-\re z^{3/2}$ if
$\arg(z)\in(-\pi,-\frac{2\pi}{3})\cup(\frac{2\pi}{3},\pi]$.
Applying a similar argument to \eqref{eq:ai-prime-asymp-positive}
and \eqref{eq:ai-prime-asymp-negative} yields the bound on $\ai'(z)$.
From the second identity in \eqref{eq:ci} one obtains the bound on $\ci_\pm(z)$ and
$\ci_\pm'(z)$. Finally, since $\bi(z)$ is a linear combination of $\ci_+(z)$
and $\ci_-(z)$, the last assertion follows.
\end{proof}

Let us recall the contours introduced in the proof of Lemma~\ref{lem:crude-asymp-eigenvalues},
\begin{equation*}
\cE^m := \left\{z\in\C: \abs{\zeta} = \bigl(m+\tfrac14\bigr)\pi\right\},
\quad
\cE_k := \left\{z\in\C: \abs{\zeta - \bigl(k-\tfrac14\bigr)\pi}=\tfrac{\pi}{2}\right\},
\quad
m,k\in\N.
\end{equation*}
In view of \eqref{eq:zeros-airy}, every $\cE_k$ encloses one and only one zero
of $\ai(-z)$, at least for $k$ sufficiently large.

\begin{lemma}
\label{lem:bound-for-gA}
There exists $m_0, k_0\in\N$ such that, for every $m\ge m_0$ and $k\ge k_0$, the
following statement holds true:
\begin{equation}
\label{eq:bound-for-gA}
\frac{g_A(-z)}{1+\abs{z}^{1/4}}
	< 8\sqrt{\pi} \abs{\text{\rm Ai}(-z)},
\end{equation}
whenever $z\in\cE^m$ or $z\in\cE_k$.
\end{lemma}
\begin{proof}
Let us begin by recalling \eqref{eq:ai-asymp-positive} and \eqref{eq:ai-asymp-negative}
in more precise terms:
\begin{align}
\label{eq:ai-also-detailed}
\ai(z)
	= \frac{e^{-\zeta}}{2\sqrt{\pi}z^{1/4}}
		\left[1+W_1(z)\right],
		\quad \abs{\arg(z)}\le\tfrac{2\pi}{3},
		\quad \abs{z}\ge 1,
\\[1mm]
\label{eq:ai-detailed}
\ai(-z)
	= \frac{1}{\sqrt{\pi}z^{1/4}}
		\left[\sin\left(\zeta+\tfrac{\pi}{4}\right)+W_2(z)\right],
		\quad \abs{\arg(z)}\le\tfrac{\pi}{3},
		\quad \abs{z}\ge 1,
\end{align}
where the functions $W_1(z)$ and $W_2(z)$ satisfy
\begin{align}
\abs{\frac{W_1(z)}{\zeta^{-1}}}\le D_1,
		\quad \abs{\arg(z)}\le\tfrac{2\pi}{3},
		\quad \abs{z}\ge 1,\nonumber
\\[1mm]
\abs{\frac{W_2(z)}{\zeta^{-1}e^{\abs{\im\zeta}}}}\le D_2,
		\quad \abs{\arg(z)}\le\tfrac{\pi}{3},
		\quad \abs{z}\ge 1.\label{eq:w2}
\end{align}

There exists $k_0\in\N$ such that, for all $k\ge k_0$, $z\in\cE_k$ implies
$\re z\ge 1$ and $\arg(z)\in(-\frac{\pi}{3},\frac{\pi}{3})$ so
$\arg(-z)\in(-\pi,-\frac{2\pi}{3})\cup(\frac{2\pi}{3},\pi]$. Since in this case
$\abss{\im z^{3/2}}=-\re (-z)^{3/2}$, one has
\begin{equation*}
\frac{g_A(-z)}{1+\abs{z}^{1/4}}
	=   \frac{e^{\abs{\im\zeta}}}{1+\abs{z}^{1/4}}
	\le \frac{e^{\abs{\im(\zeta+\frac{\pi}{4})}}}{\abs{z}^{1/4}}
\end{equation*}
for all $z\in\cE_k$ and $k\ge k_0$. By a well-known result
(see \cite[Ch.~2, Lemma 1]{poeschel}),
\begin{equation*}
\abs{w-n\pi}\ge \frac{\pi}{4} \implies e^{\abs{\im w}}<4\abs{\sin w}
\end{equation*}
for all integer $n$. Hence,
\begin{equation}
\label{eq:bound-1}
\frac{g_A(-z)}{1+\abs{z}^{1/4}}
	< 4 \frac{\abs{\sin(\zeta+\frac{\pi}{4})}}{\abs{z}^{1/4}}
\end{equation}
for all $z\in\cE_k$ and $k\ge k_0$.
On the other hand, since $\abs{\sin(\zeta+\tfrac{\pi}{4})}\ge d>0$ for all $z\in\cE_k$
and $k\ge k_0$, \eqref{eq:ai-detailed} implies
\begin{equation*}
\abs{\ai(-z)}
	\ge \frac{\abs{\sin(\zeta+\frac{\pi}{4})}}{\sqrt{\pi}\abs{z}^{1/4}}
	\abs{1-\frac{\abs{W_2(z)}}{\abs{\sin(\zeta+\tfrac{\pi}{4})}}}.
\end{equation*}
However,
\begin{equation*}
\frac{\abs{W_2(z)}}{\abs{\sin(\zeta+\tfrac{\pi}{4})}}
	\le \frac{e^{\abs{\im\zeta}}}{\abs{\zeta}}\frac{D_2}{d},
\end{equation*}
and note that $\abs{\im\zeta}\le \pi/2$ if $z\in\cE_k$.
Thus, by increasing $k_0$ if necessary, we have
\begin{equation}
\label{eq:bound-2}
\abs{\ai(-z)}
	\ge \frac{\abs{\sin(\zeta+\frac{\pi}{4})}}{2\sqrt{\pi}\abs{z}^{1/4}},
\end{equation}
for all $z\in\cE_k$ with $k\ge k_0$.

The proof concerning $\cE^m$ is analogous: Suppose $m_0=k_0$. Then, by the previous
argument, \eqref{eq:bound-for-gA} holds for $z\in\cE^m$ within the sector
$\arg(z)\in[-\frac{\pi}{3},\frac{\pi}{3}]$, for $m\ge m_0$. Within the sector
$\arg(-z)\in[-\frac{2\pi}{3},\frac{2\pi}{3}]$, we have ($\eta:=\frac23(-z)^{3/2}$)
\begin{equation*}
\frac{g_A(-z)}{1+\abs{z}^{1/4}}
	\le \frac{e^{-\re\eta}}{\abs{z}^{1/4}}
\end{equation*}
and, due to \eqref{eq:ai-also-detailed},
\begin{equation*}
\abs{\ai(-z)}
	\ge \frac{e^{-\re\eta}}{2\sqrt{\pi}\abs{z}^{1/4}}
		\abs{1-\abs{W_1(-z)}}.
\end{equation*}
Finally, using \eqref{eq:w2} ---and increasing $m_0$ if required---,
we have $1-\abs{W_1(-z)}\ge 1/4$ whenever $\abs{z}\ge m_0$.
\end{proof}


\end{document}